\documentclass[a4paper,12pt]{amsart}
\usepackage[utf8]{inputenc}

\usepackage[T1]{fontenc}
\usepackage[english]{babel}
\usepackage{color}
\usepackage{amsmath}
\usepackage{amssymb}
\usepackage{amsthm}
\usepackage{graphicx}
\usepackage[colorlinks=true, linkcolor=blue]{hyperref}

\usepackage{geometry}

\usepackage{lmodern}
\usepackage{microtype}

\usepackage{subfiles}

\usepackage{soul}

\usepackage[capitalise]{cleveref}

\usepackage[colorinlistoftodos,prependcaption,textsize=tiny]{todonotes}

\newcommand {\R} {\mathbb{R}}
\newcommand {\Z} {\mathbb{Z}}
\newcommand {\T} {\mathbb{T}}
\newcommand {\N} {\mathbb{N}}

\newcommand {\p} {\partial}
\newcommand {\dt} {\partial_t}

\newcommand{\calD}{\mathcal{D}}

\newcommand{\calG}{\mathcal{G}}

\newcommand{\calI}{\mathcal{I}}

\newcommand{\calR}{\mathcal{R}}

\newcommand{\eps}{\varepsilon}

\DeclareMathOperator{\dive }{div}

\theoremstyle{plain}
\newtheorem{thm}{Theorem}
\newtheorem*{thm*}{Theorem}
\newtheorem{prop}{Proposition}

\newtheorem*{cor*}{Corollary}

\theoremstyle{plain}
\newtheorem{theorem}{Theorem}
\newtheorem{lemma}[theorem]{Lemma}

\theoremstyle{definition}

\newtheorem{remark}[theorem]{Remark}

\date{\today}

\begin{document}
\title[Non-resistive MHD]{Large norm inflation of the current in the viscous, non-resistive magnetohydrodynamics equations}
\begin{abstract}
	We consider the ideally conducting, viscous
	magnetohydrodynamics (MHD) equations in two
	dimensions. Specifically, we study the nonlinear dynamics near a combination of Couette flow and a constant magnetic field in a periodic infinite channel.
	In contrast to the Navier-Stokes equations this system is shown to exhibit
	algebraic instability and large norm inflation of the magnetic current on
	non-perturbative time scales. 
\end{abstract}
\author{Michele Dolce}
\author{Niklas Knobel}
\author{Christian Zillinger}

\maketitle

\setcounter{tocdepth}{1}
\tableofcontents

\section{Introduction}
\label{sec:introduction}
In this article, we study the long-time behavior of the non-resistive viscous
magnetohydrodynamics (MHD) equations in a two-dimensional infinite periodic
channel:
\begin{align}
	\label{eq:nrMHD}
	\begin{split}
		\p_t V + V\cdot\nabla V + \nabla \Pi & =\nu \Delta V+  \dive (\nabla^\perp \Phi \otimes   \nabla^\perp \Phi ), \\
		\p_t \Phi +  V\cdot\nabla \Phi       & =0,                                                                    \\
		\text{div}(V)                   & =0,                                                                    \\
		(t,x,y)                         & \in \R_{+}\times \T\times \R =: \Omega.
	\end{split}
\end{align}
These equations model the evolution of an electrically conductive fluid with
velocity $V:\Omega\rightarrow \R^2$ interacting with a magnetic field
$B:=\nabla^{\perp} \Phi: \Omega \rightarrow \R^2$. The function
$\Phi:\Omega\rightarrow \R$ denotes the magnetic potential and $\Pi:\Omega
\rightarrow \R$ denotes the pressure. Common settings described by these
equations include the modeling of plasmas or molten metals in industrial
applications \cite{davidson2016introduction}. The MHD equations are a subject of
highly active research and include several distinct partial dissipation regimes.
More precisely, these equations include both dissipation in the fluid component
with a viscosity parameter $\nu > 0$, as well as resistivity of magnetic
component $\mu \geq 0$. The well-posedness of the non-resistive/ideal conductor
case, $\mu=0$, studied in this article is a challenging open problem in the
field \cite{lin2012some}. In contrast to the two-dimensional Euler or
Navier-Stokes equations, it is not known whether the equations can exhibit
finite time blow-up. In particular, we mention recent results on local
well-posedness \cite{ji2019stability}, the stabilizing effect of large, constant
magnetic fields
\cite{lin2015global,ren2014global,zhao2023asymptotic,wu2021global,wei2017global}
(for small velocities) and improved stability estimates for small velocities
\cite{cobb2023elsasser,cobb2021symmetry}.

In contrast, this article focuses on the interaction of mixing by a large
(shear) velocity and the (de)stabilizing role of the magnetic field. To this
end, we study the nonlinear stability problem near the stationary solution
\begin{align}
	\label{eq:stationarysol}
	V(x,y,t)=(y,0), \ \Phi(x,y,t)=-\alpha y, \ \nabla^{\perp} \Phi(x,y,t)=(\alpha, 0),
\end{align}
where $\alpha \in \R$ denotes the size of the constant magnetic field.

In view of the shear flow in the stationary solution \eqref{eq:stationarysol},
it is convenient to work with the following perturbative unknowns in coordinates
moving with the flow
\begin{align*}
	v(x,y,t)     & = V(x+yt,y,t)- y e_1          \\
  w(x,y,t)    &= (\nabla^\perp \cdot V)(x+yt,y,t)+1 \\
	\phi (x,y,t) & = \Phi (x+yt,y,t) + \alpha y,\\
  b(x,y,t)&=(\nabla^\perp_t \phi)(x,y,t)\\
  \Delta_t(\psi,\phi)(x,y,t)&=(w,j)(x,y,t), \qquad
\end{align*}
where $\Delta_t=\p_{x}^2+(\p_y-t\p_x)^2, \nabla_t=(\p_x,\p_y-t\p_x)$ denote the
differential operators in coordinates moving with the shear flow. We also introduce the projection operators
\begin{equation*}
  f_0(y):=(\mathbb{P}_0f)(y)=\frac{1}{2\pi}\int_{\mathbb{T}}f(x,y)d x, \qquad f_{\neq}:=\mathbb{P}_{\neq}f=(I-\mathbb{P}_0)f.
\end{equation*}
It is then useful to isolate the \textit{nonlinear shear flow} component by
observing that
\begin{equation*}
  v(x,y,t)=(\nabla^\perp_t\Delta_t^{-1}w)(x,y,t)=v_0^x(y,t)e_1+(\nabla^\perp_t\psi_{\neq})(x,y,t).
\end{equation*}
In terms of the unknowns in the moving frame, the non-resistive MHD equations
\eqref{eq:nrMHD} can be equivalently expressed as
\begin{align}
	\label{eq:NL}
	\begin{split}
		\partial_t w  & =  \nu  \Delta_t w+ \alpha \partial_x \Delta_t \phi   +  (\nabla^\perp \phi \cdot \nabla) \Delta_t  \phi - v\cdot \nabla_t w, \\
		\p_t \phi                                          & =\alpha \p_x \Delta_t^{-1} w-v\cdot \nabla_t \phi,\\
    \p_tv_0^x&= \nu \p_{y}^2v_0^x+( b_{\neq} \cdot\nabla_t b_{\neq}^x -v_{\neq} \cdot\nabla_t v^x_{\neq})_0,\\
    w|_{t=0}&=w_{in}, \quad \phi|_{t=0}=\phi_{in}.
	\end{split}
\end{align}
where we identify inverse powers of differential operators with their symbols in
the Fourier space. As we discuss in greater detail in Section \ref{sec:linear},
the viscous dissipation and shear dynamics break the symmetry of the equations
and link the evolution of the streamfunction and magnetic potential in an
  intricate way. This link is
encoded in the \emph{good unknown}
\begin{align}
	\label{def:G}
 G(x,y,t) :=\nu \psi_{\neq}(x,y,t)+ \alpha\p_x \Delta_t^{-1}\phi_{\neq}(x,y,t),
\end{align}
and the equations \eqref{eq:NL} are rewritten in terms of $G,\phi$ in Section \ref{sec:linear}.

Our main goal in the following is to characterize the nonlinear Lyapunov
stability of the non-resistive MHD equations in terms of $(\phi, G)$ and we are particularly interested in possible destabilizing effects of a large viscosity. For this reason, from now on we fix $$\alpha=\nu=1.$$
The
stability properties turn out to drastically change depending on the time-scale
under consideration. Indeed, already from a first inspection of the system \eqref{eq:NL} we see that in the equation for $w$ there is a  nonlinear term involving factor scaling as $O(t^2)$, which is extremely dangerous since we have no available dissipation of $\phi$. However, we note that for small perturbations, on a sufficiently
short time-scale the nonlinear evolution can be treated perturbatively.
\begin{prop}[Perturbative time-scale]
	\label{prop:perturbative}
    Let $\alpha=1=\nu$. Then for every $N\geq 5$ there exists $\epsilon_0>0$ such that for any
  $0<\epsilon<\epsilon_0$ and any mean free initial data with
	\begin{align*}
		\|\phi_{in}\|_{H^N} + \|G_{in}\|_{H^N} < \epsilon,
	\end{align*}
	the corresponding solution satisfies the bound
	\begin{align*}
		\|\phi(t)\|_{H^N} + \|G(t)\|_{H^N} \lesssim \epsilon
	\end{align*}
	for all times $0<t \ll \epsilon^{-1/2}$.
\end{prop}
We say that the equations are stable in Sobolev regularity on the perturbative
time-scale $\mathcal{O}(\epsilon^{-1/2})$. More generally, this stability
extends also to higher regularity, e.g. analytic or Gevrey regularity, instead
of the Sobolev regularity $H^N$ (see Section \ref{sec:perturbative}).
\begin{remark}
    To understand why $\mathcal{O}(\epsilon^{-1/2})$ is the natural perturbative time-scale, consider the nonlinear term $\nabla^\perp\phi\cdot \nabla \Delta_t\phi$ in the equation for $w$. The good unknown $G$ involves $\Delta_t^{-1}w$ and  one needs to understand the size of  $\Delta_t^{-1}(\nabla^\perp\phi\cdot \nabla \Delta_t\phi)$ when integrated in time. Exploiting the fluid dissipation and symmetries, it is  enough to estimate $\Delta_t^{-1}(\nabla^\perp\phi\cdot \nabla \nabla_t\phi)$. Now, the differential operators are anisotropic in time-frequency space. Therefore the worst-case scenario is when one loses time growth from $\nabla_t$ and one of the differential operators $\nabla$ or $\nabla^\perp$ is at high frequency, where one can replace a derivative by a factor of $\mathcal{O}(t)$. This makes the dangerous term of size $\mathcal{O}(\eps^2 t^2)$ after ``integrating in time'' the time-decay of $\Delta^{-1}_t$. Therefore, to treat the nonlinearity perturbatively one needs the time scale restriction $t^2\eps \ll 1$.
\end{remark}
On longer time scales, nonlinear effects become dominant and require a
tailor-made analysis. In particular, to handle infinite regularity losses on longer time-scales we need to work in Gevrey spaces where the associated norm is defined as
\begin{equation*}
    \|f\|^2_{\mathcal{G}^{s,\lambda}}:=\sum_{k\in \mathbb{Z}}\int_{\mathbb{R}} e^{2\lambda(|k|+|\eta|)^s}|\mathcal{F}(f)(k,\eta)|^2 d\eta,
\end{equation*}
where $\mathcal{F}$ is the Fourier transform. A precise, quantitative decrease of $\lambda$ with time then allows to push beyond the perturbative time scale. This requires the use of a weight adapted to the main nonlinear \textit{echoes} which we discuss more precisely. Our main results are summarized in the following theorem.
\begin{thm}
\label{thm:main}
    Consider the non-resistive MHD equations \eqref{eq:NL} with
    $\alpha=1=\nu$. Then for every $\tfrac 13< s\le 1  $, there exists a
    threshold $\epsilon_0>0$ and constants $\lambda_1\ge \lambda_2 >0$, such
    that the evolution satisfies the following stability estimate: for any $0<\eps\le \epsilon_0$ and any initial data which satisfies the bound
    \begin{align*}
        \Vert v^x_{0,in}\Vert_{\calG^{s,\lambda_1} } +\Vert \phi_{in}\Vert_{\calG^{s,\lambda_1} } +\Vert G_{in}\Vert_{\calG^{s,\lambda_1} }&\le \eps,
    \end{align*}
    the corresponding solution satisfies the energy estimates
    \begin{align*}
      \Vert \langle \p_y \rangle^{-1} v^x_0(t)\Vert_{\calG^{s,\lambda_2} }&\lesssim \eps, \\
      \Vert \phi(t)   \Vert_{\calG^{s,\lambda_2}} +\Vert G(t) \Vert_{\calG^{s,\lambda_2}}&\lesssim \eps,
    \end{align*}
    for all times $0\le t^{\frac 32}\le \delta \eps^{-1}$.

Furthermore, there exist $K,c,k_0>0$ such for $\chi(k)=\mathbf{1}_{|k|\geq k_0}$ and initial data satisfying also $\Vert \chi  \phi_{in} \Vert_{H^{-2}}\ge \max (K \delta \eps, \Vert \chi\p_x  G \Vert_{H^{-2}} , c\Vert   \phi_{in} \Vert_{H^{-2}})$, the current density and magnetic field satisfy the instability estimates
\begin{align}\label{eq:jinst}
  \begin{split}
    \min\left(\Vert j(t)\Vert_{L^2},\langle t \rangle  \Vert b(t)\Vert_{L^2}\right)\gtrsim \langle t\rangle^2 \Vert   \phi_{in} \Vert_{H^{-2}}.
\end{split}
\end{align}
In particular,  for $T= \delta \eps^{-\frac23}$ one has the large norm-inflation $\Vert j(T)\Vert_{L^2}\gtrsim \delta^{3}\eps^{-\frac 1 3 }$.
\end{thm}

Let us comment on these results:
\begin{itemize}
\item The instability is driven by nonlinear resonances, reminiscent of plasma
  echoes \cite{rae1984temporal}, acting on $G$. As we discuss in Section
  \ref{sec:phi}, for a model problem with $G=0$ stability can be be established
  on a longer time scale $t<\eps^{-1}/|\log(\eps)|$ and under weaker
  regularity requirements. Thus, the instability is driven by the good unknown
  $G$. This is in contrast to the stability of good unknowns in other systems
  such as the partially dissipative Boussinesq equations
  \cite{bedrossian2023nonlinear,masmoudi2022stability,masmoudi2023asymptotic},
  where the submodel with $G=0$ already includes the main resonance mechanism.
\item In view of the \emph{current instability} \eqref{eq:jinst}, stability past time scales $\mathcal{O}(\eps^{-\frac 23 })$ cannot be expected even for analytic initial data (see Section \ref{sec:perturbative} for a discussion).
  We emphasize that we have a full control of the solution far beyond the perturbative
  time scale of Proposition \ref{prop:perturbative}.
  \item The strong algebraic growth instability \eqref{eq:jinst} as well as strong resonances in the nonlinear terms and an associated lack of time integrability pose major, new mathematical challenges and exhibit qualitatively different behavior than in the (anisotropic) resistive settings studied in \cite{Dolce,knobel2023sobolev}.
  \item In the ideal case \cite{NiklasMHD2024}, when no dissipation is present, the second author obtains stability in Gevrey spaces up until times $t\sim \eps^{-1}$. Here
    constant magnetic field yields an interaction between velocity and magnetic
    field. Therefore, the decay of the velocity field (inviscid damping) and
    growth of the current (by a dynamo effect) balance such that velocity and
    magnetic field only grow by a finite amount. Therefore, the viscosity and
    the resulting symmetry breaking destabilize the equation on time scales $t\sim \eps^{-\frac 23 }$.
   \item In contrast, for the resistive and inviscid case \cite{knobel2023echoes,zhao2023asymptotic} there is a reverse symmetry breaking, which leads to inviscid damping in the velocity field. Therefore, in that case, one obtains global in-time stability in Gevrey spaces.
  \item When $\alpha,\nu$ are not set to be equal to $1$, it is possible to obtain a stability threshold, namely $\eps$ scaling as a suitable power of $\nu$. This requires some technical adjustments of the analysis performed in this paper. In particular, one needs to account for a possible transient growth for $G$ of size $\mathcal{O}(\nu^{-\frac23})$, which can be handled with Fourier multipliers which are by now standard in the literature, e.g. \cite{bedrossian17}. Moreover, for small $\nu$ we also expect to extend the time-scale in view of \cite{NiklasMHD2024}. It would be interesting to connect the results available in the ideal \cite{NiklasMHD2024} and fully dissipative cases \cite{liss2020sobolev, Dolce, knobel2024sobolev, wang2024stabilitythreshold2dmhd} in the limiting regimes of interest. This goes way beyond the scope of this paper and therefore we only present the results in the relevant example of large viscosity $\nu=1$ and $\alpha=1$.
\end{itemize}

The remainder of this article is structured as follows:
\begin{itemize}
\item In Section \ref{sec:linear} we establish linear stability results. In
  particular, this analysis allows us to identify an optimal choice of unknowns.
  Due to the strong frequency decoupling of the linearized equations, these
  equations exhibit strong stability properties in arbitrary regularity (which is an oversimplification compared to the nonlinear dynamics).
\item In Section \ref{sec:perturbative} we establish stability on the
  perturbative time scale $\eps^{-1/2}$. Furthermore, we discuss which nonlinear terms can become large after this time scale and provide a heuristic discussion of the physical resonance mechanisms. These considerations motivate the construction of energy functionals in the following sections.
\item Section \ref{sec:bootstrap} establishes nonlinear stability on
  non-perturbative time scales. Here, we use a bootstrap approach to control
  adapted Lyapunov functionals. In order to highlight the nonlinear resonance
  mechanisms, we individually discuss resonances in the magnetic and fluid parts
  of the equations in individual subsections.
  \item In Section \ref{sec:Instability} we prove the norm-inflation result \eqref{eq:jinst}.
\end{itemize}

\section{Linear Stability and Adapted Unknowns}
\label{sec:linear}
In this section, we consider the linearized stability problem around the
stationary solution \eqref{eq:stationarysol} and consider the case
$\alpha=\nu=1$. Expressing the nonlinear MHD equations \eqref{eq:NL} in terms of
$\phi$ and $G$ (defined in \eqref{def:G}), we obtain the system
\begin{align}
  \label{eq:phiGeq}
 	\begin{split}
    \p_t G &= (\Delta_t+ 2 \p_x \p_y^t\Delta_t^{-1} + \p_x^2 \Delta_t^{-1}) G- \p_x^3 \Delta_t^{-2} \phi_{\neq}\\
           &\quad + \Delta^{-1}_t \nabla^\perp_t \cdot(( \nabla^\perp \phi  \cdot\nabla) \nabla^\perp_t  \phi )_{\neq}-\Delta^{-1}_t  ( v\cdot \nabla_t w)_{\neq} -   \p_x\Delta_t^{-1}(v\cdot\nabla_t  \phi)_{\neq} \\
    \p_t \phi   &=- \p_x^2 \Delta_t^{-1} \phi_{\neq} + \p_x G \\
           &\quad - ( \nabla^\perp G\cdot \nabla \phi) + (\p_x\nabla^\perp \Delta_t^{-1} \phi_{\neq}\cdot \nabla) \phi - v^x_{0}\p_x \phi_{\neq}  \\
    \p_t v^x_{0}&= \p_y^2 v^x_{0} +( b_{\neq} \cdot\nabla_t b_{\neq}^x -v_{\neq} \cdot\nabla_t v^x_{\neq})_0.
	\end{split}
\end{align}
In order to investigate this system's stability, in this section, we study the
associated linearized problem:
\begin{align}
	\label{eq:linearizedeq}
	\begin{split}
		\p_t G    & = ( \Delta_t+2  \p_x \p_y^t\Delta_t^{-1} +  \p_x^2 \Delta_t^{-1}) G - \p_x^3 \Delta_t^{-2} \phi,                                      \\
		\p_t \phi & =- \p_x^2 \Delta_t^{-1} \phi + \p_x G.
	\end{split}
\end{align}
Since the underlying stationary solution \eqref{eq:stationarysol} is affine, the
linearized equations \eqref{eq:linearizedeq} decouple in frequency. Indeed,
applying a Fourier transform in both $x\in \T$ and $y\in \R$, the system
\eqref{eq:linearizedeq} decouples with respect to the frequencies $(k,\eta) \in
\Z \times \R$ and reads
\begin{align}
  \label{eq:decoupledsystem}
  \begin{split}
		\dt \begin{pmatrix}
      \mathcal{F} G \\ \mathcal{F} \phi
    \end{pmatrix}
		=
		\begin{pmatrix}
			-(k^2+(\eta-kt)^2 +2  \frac{k(\eta-kt)}{k^2+(\eta-kt)^2} +  \frac{k^2}{k^2+(\eta-kt)^2} & \frac{ik^3}{(k^2+(\eta-kt)^2)^2} \\
      ik   &  -\frac{k^2}{k^2+(\eta-kt)^2}
		\end{pmatrix}
		\begin{pmatrix}
			\mathcal{F}G \\ \mathcal{F}\phi
		\end{pmatrix}.
  \end{split}
\end{align}
Here, with slight abuse of notation, we again denote the Fourier transformed
functions by $\phi$ and $G$. On the one hand, this decoupling allows for a very
transparent analysis and strong stability properties. On the other hand, the
dynamics ``oversimplify''; the stability properties of the nonlinear equations
are much more restrictive due to resonance effects which are omitted in the
linearization (see Section \ref{sec:resonance} for a heuristic discussion of the
resonances).

In view to later applications to the nonlinear model, we establish stability in
terms of an energy weighted with a Fourier multiplier $m_{L}(t,D)$. In the present
setting, due to the strong decoupling, this multiplier can be constructed in a
straightforward way as follows 
\begin{align}
\label{def:mL}
     \begin{split}&\frac {\p_t m_L } {m_L} =
                                \begin{cases}
                                 \tfrac{5}{1+(t-\frac \eta k )^2 } & \text{ if } \vert k,\eta\vert\le 10\langle t \rangle, \\
                                  0 & \text{ else},
                                \end{cases}\\
      &    m_L(0,k,\eta)= \langle k,\eta\rangle^{-N}, 
      \end{split}
\end{align}
with $N\geq 1$. For
the construction in the nonlinear setting, see Section \ref{sec:weights}.

The linear stability results for the system \eqref{eq:decoupledsystem} are
summarized in the following proposition.
\begin{prop}
	\label{prop:linear}
  Let $A:\R_{+}\times \Z \times \R \rightarrow \R_{+}$ 
  be a smooth, positive Fourier multiplier, which satisfies 
  \begin{align}
\label{eq:Alin}
      A\geq m_L^{-1}, \qquad \frac {\p_t A } {A}\leq -\frac{\p_t m_L}{m_L} .
  \end{align}
  Then, for any solution
  $(G,\phi)$ of \eqref{eq:linearizedeq} the energy functional
  \begin{align}
    E(t) = \frac 12  \|A  G(t)\|_{L^2}^2 +  \frac 12  \|A  \phi(t)\|_{L^2}^2 
  \end{align}
  satisfies the energy estimate 
  \begin{align}
     E(t) + \frac 12  \int_0^t \Vert \nabla_t A  G  \Vert_{L^2}^2+  \Vert \p_x\Lambda_t^{-1} A  \phi \Vert_{L^2}^2 d\tau  \leq E(0). 
  \end{align}
  In particular, the linear error term
    \begin{align*}
L_{G,\phi}:=\, \langle (2 \p_x \p_y^t\Delta_t^{-1} +  \p_x^2 \Delta_t^{-1}) AG, AG\rangle -\langle \big(\p_x^3 \Delta_t^{-2}+\p_x\big)A\phi_{\neq} , AG\rangle.
    \end{align*}
    satisfies the bound
    \begin{align}
    \label{eq:linearestimate}
      |L_{G,\phi}| \le \frac 12 \Vert \nabla_t A  G \Vert_{L^2}^2 +\frac 12  \Vert\p_x \Lambda_t^{-1} A  \phi \Vert_{L^2}^2+ \frac{1}{2} \Vert\sqrt{\tfrac {\p_t m_{L} } {m_{L}} } A  \phi \Vert_{L^2}^2.
  \end{align}
     \end{prop}

 \begin{remark}
Choosing $A=m_L^{-1}$, we see that $E(t)$ is comparable to $\|(G,\phi)\|_{H^N}^2$ and hence that the linearized MHD equations are stable in arbitrary Sobolev regularity.

  In Section \ref{sec:weights} we construct a specific Fourier multiplier $A$
  which incorporates $m_{L}^{-1}$ and is
  additionally tailored to control nonlinear effects and to satisfy good
  commutator estimates. The estimates of this proposition also hold for that
  weight as shown in Lemma \ref{lem:LGphi}.

  Throughout this article we have normalized viscosity as $\nu=1$. For much larger
  viscosities, e.g. $\nu\geq 10$, linear stability estimates become trivial,
  since all possible growth mechanisms can easily be absorbed by viscous
  dissipation. In the case of small viscosities, it is well-known that the term
  $(-\nu(k^2+(\eta-kt)^2)+
  2\frac{\tfrac{\eta}{k}-t}{1+(\tfrac{\eta}{k}-t)^2})\mathcal{F}G$ can lead to
  norm inflation by a factor $(\nu k^2)^{-2/3}$. As mentioned in Section
  \ref{sec:introduction}, the results of this article can be extended to this
  case with some technical effort and will yield some nonlinear stability
  threshold in terms of $\nu>0$. Since we do not aim for optimality of that
  threshold, we do not pursue this extension in the present article.
\end{remark}

\begin{proof}[Proof of Proposition \ref{prop:linear}]
  We recall that all operators in the linearized MHD equations
  \eqref{eq:linearizedeq} are given by Fourier multipliers and that the
  equations hence decouple in Fourier space.
  Therefore, it suffices to show that for any given frequency $k,\eta \in \Z
  \times \R$, the following frequency-wise bound holds
      \begin{align}
\label{eq:lindec}
    \begin{split}
      & \qquad \frac 12 \dt \left( |AG|^2 + |A\phi |^2   \right)  \\
      &\leq  - \frac 12 (k^2 + (\eta-kt)^2) |AG|^2 - \frac 1 2 \frac{k^2}{k^2+(\eta-kt)^2} |A\phi|^2  - \frac {\p_t m_{L} }{m_{L}} |(AG, A\phi)|^2.
    \end{split}
      \end{align}
      Since Fourier multipliers commute and $\frac{\dt A}{A}\leq - \frac{\dt
        m_{L}}{m_{L}}$, we only need to compute the explicit time-derivatives:
    \begin{align*}
     \frac 12 \dt (|G|^2+ |\phi|^2) &\leq - (k^2 + (\eta-kt)^2) |G|^2 - \frac{k^2}{k^2+(\eta-kt)^2} |\phi|^2 \\
               & \quad + \frac{|k(\eta-kt)|}{k^2+(\eta-kt)^2} |G|^2 + \frac 1 2 \left(- \left(\frac{k^2}{k^2+(\eta-kt)^2}\right)^2+ 1 \right) |\phi| |G|\\
               &\le - (k^2 + (\eta-kt)^2) |G|^2 -\frac 1 2  \frac{k^2}{k^2+(\eta-kt)^2} |\phi|^2 \\
               & \quad + \frac{|k(\eta-kt)|}{k^2+(\eta-kt)^2} |G|^2 + \frac {(1+(t- \frac \eta k )^2 }8 \left\vert \left( \frac{k^2}{k^2+(\eta-kt)^2}\right)^2- 1 \right\vert  |G|^2. 
    \end{align*}
    Using this, it follows that \eqref{eq:lindec} holds if
\begin{align*}
     \frac{|k(\eta-kt)|}{k^2+(\eta-kt)^2}  + \frac {(1+(t- \frac \eta k )^2 }8 \left\vert \left( \frac{k^2}{k^2+(\eta-kt)^2}\right)^2- 1 \right\vert \le \frac 1 2 (k^2 + (\eta-kt)^2)- \frac {\p_t m_L }{m_L},
\end{align*}
which is the case by the definition of $m_L$ \eqref{def:mL} and our assumed bounds on  $A$.
\end{proof}
These linear stability results are necessary preconditions for nonlinear
stability but turn out to not be sufficient. In the following section, we
introduce nonlinear instability mechanisms and show that they lead to large
instabilities on any non-perturbative time-scale.

\section{Nonlinear Resonances and Short Time Stability}
\label{sec:perturbative}
In contrast to the linear stability results of Section \ref{sec:linear}, the
nonlinear dynamics exhibit resonance phenomena which can result in large norm
inflation and loss of stability. These resonances are reminiscent of ``plasma
echoes'' in the Vlasov-Poisson equations
\cite{malmberg1968plasma,rae1984temporal,bedrossian2016nonlinear,zillinger2020landau}
or ``fluid echoes'' in the Euler equations
\cite{yu2005fluid,dengZ2019,dengmasmoudi2018}. However, for the present setting
of the MHD equations the mechanism exhibits a much stronger time dependence,
strong coupling effects (e.g., the resonances are not suppressed by viscous
dissipation) and new time scales.

As a first result and in order to establish a proper context, we employ basic
well-posedness estimates to establish stability in Sobolev regularity (or any
higher regularity class) up to a perturbative time scale
$\mathcal{O}(\eps^{-1/2})$. More precisely, we show that on this time scale
the energy constructed in Section \ref{sec:linear} remains bounded also under
the nonlinear evolution.
\begin{lemma}
  \label{lem:perturbative}
  Let $N\in \N$ with $N\geq 5$ and let $m_L(t,D)$ be the Fourier multiplier defined in \eqref{def:mL}. Then there exists $\eps_0=\eps_0(N)>0$
  such that for all $0<\eps<\eps_0$ and any mean free initial data
  $(\phi_{in},G_{in}) \in H^N$ with
  \begin{align*}
    \|\phi_{in}\|_{H^N} + \|G_{in} \|_{H^N}< \eps
  \end{align*}
  the corresponding solution $(\phi,G)$ of the nonlinear MHD equation
  \eqref{eq:NL} satisfies
  \begin{align}
    \|\phi(t)\|_{H^N} + \|G(t)\|_{H^N} \lesssim \|m_L(t,D)(\phi,G)(t)\|_{L^2} \lesssim \eps
  \end{align}
  for all times $0<t \ll \eps^{-1/2}$.
\end{lemma}
This result in particular implies the stability estimate of Proposition
\ref{prop:perturbative}.

\begin{proof}[Proof of Lemma \ref{lem:perturbative}]
   We interpret the nonlinear MHD equations \eqref{eq:phiGeq} as
  a forced linear problem
  \begin{align*}
    \dt
    \begin{pmatrix}
      G \\ \phi
    \end{pmatrix}
   = L
    \begin{pmatrix}
      G \\ \phi
    \end{pmatrix}
    + F,
  \end{align*}
  and aim to establish an a priori bound on a suitable energy. As in Proposition \ref{prop:linear}, we set   $A=m_{L}^{-1}$ and hence, defining the energy
  \begin{align*}
    E(t):= \frac 12  \|A  G(t)\|_{L^2}^2 +  \frac 12  \|A  \phi(t)\|_{L^2}^2,
  \end{align*}
  we know that it satisfies satisfies the bound
  \begin{align*}
    \dt E(t)\leq \langle A (G,\phi), A F \rangle -  \Vert \nabla_t A  G \Vert_{L^2}^2 -\frac 12  \Vert\p_x \Lambda_t^{-1} A  \phi \Vert_{L^2}^2- \frac{1}{2} \Vert\sqrt{\tfrac {\p_t m_{L} } {m_{L}} } A  \phi \Vert_{L^2}^2.
  \end{align*}
  We now claim that the contribution of the nonlinear forcing,
  $\langle A (G,\phi), A F \rangle$, can be controlled in such a way that
  \begin{align}
    \label{eq:perturbclaim}
    \dt E(t) +\Vert A \nabla_t   G \Vert_{L^2}^2\leq  C (1+t) E(t)^{3/2} + C\langle t \rangle \|A\nabla_t G\| E(t).
  \end{align}
  Using the dissipation of $G$ and the bootstrap assumption that $E(t)\lesssim
  \eps^2$, it follows that
  \begin{align*}
    E(T)- E(0) \leq \int_0^T \dt E dt \leq C (T+T^2)\eps^3.
  \end{align*}
Thus, if $\eps (T+T^2)$ is sufficiently small we can improve the bootstrap assumption $E(t)\lesssim \eps^2$ and show that this holds true on the whole interval $[0,T]$, thus proving the desired result since we are interested in the case $T\ll \eps^{-1/2} $.

  To control the nonlinear forcing, we recall from \eqref{eq:NL} that $F$ consists of several quadratic
  nonlinearities:
  \begin{align*}
    F& = \begin{pmatrix}
      \Delta^{-1}_t \nabla^\perp_t \cdot(( \nabla^\perp \phi  \cdot\nabla) \nabla^\perp_t  \phi )_{\neq}-\Delta^{-1}_t  ( v\cdot \nabla_t w)_{\neq} -   \p_x\Delta_t^{-1}(v\cdot\nabla_t  \phi)_{\neq} \\
      - ( \nabla^\perp G\cdot \nabla \phi) + (\p_x\nabla^\perp \Delta_t^{-1} \phi_{\neq}\cdot \nabla) \phi - v^x_{0}\p_x \phi_{\neq}
    \end{pmatrix}.
  \end{align*}
  Before discussing each term, we make the following observations:
  \begin{itemize}
    \item The nonlinear forcing by $\Delta^{-1}_t \nabla^\perp_t \cdot(( \nabla^\perp \phi  \cdot\nabla) \nabla^\perp_t  \phi )_{\neq}$ increases in time as $\mathcal{O}(t^2)$. Therefore one can a priori not expect a better bound than \eqref{eq:perturbclaim} unless one considers energies which lose regularity in time (or rather lose in the radius of convergence, see Sections \ref{sec:resonance} and \ref{sec:bootstrap}). 
 % \item \textcolor{red}{OLD} The operator norm of the order zero operators contained in $\phi \mapsto \p_x\nabla^\perp \Delta_t^{-1}
 %   \phi_{\neq}$ grows with time as $\mathcal{O}(t)$. Thus one can a priori not
 %   expect a better bound than \eqref{eq:perturbclaim} unless one considers
 %   energies which lose regularity in time (or rather lose in the radius of
 %   convergence, see Sections \ref{sec:resonance} and \ref{sec:bootstrap}).\footnote{I would focus here already more on the operator $\Delta^{-1}_t \nabla^\perp_t \cdot(( \nabla^\perp \phi  \cdot\nabla) \nabla^\perp_t  \phi )_{\neq}$, since this term gets 'first' worse. }
  \item Since we consider equations with \emph{partial dissipation}, any terms
    involving derivatives of $\phi$ need to exploit cancellation and commutators to avoid losses
    of derivatives.
    In particular, we need to exploit commutation properties of the
    multiplier $A$. In the perturbative regime considered in this section, this
    does not yet require detailed optimizations but will pose a major challenge
    to the construction of $A$ on longer time scales considered in Section \ref{sec:bootstrap}.
  \item Since $N\geq 5$, we may make use of Sobolev embeddings and paraproduct
    decompositions (this will also feature prominently in the following Section
    \ref{sec:resonance}). In particular, a main focus will be on products where
    one factor is frequency localized at ``low'' frequencies, while the other is
    at ``high'' frequencies.
  \item The time-dependent Fourier multiplier $\Delta_t^{-1}\leadsto
    \frac{1}{k^2+(\eta-kt)^2}$ always provides a gain with respect to $x$
    derivatives. For $y$ derivatives it also provides a strong gain unless
    $|\eta-kt|\leq \frac{|kt|}{2}$. It is in that frequency region where the
    loss of a factor $t$ as in the claimed estimate \eqref{eq:perturbclaim} occurs.
  \end{itemize}

  Since we establish improved estimates for all terms (additionally using Gevrey
  regularity) in Section \ref{sec:bootstrap}, in this proof we only discuss the
  most important terms in detail. We follow the same
  notational conventions as used in Section \ref{sec:bootstrap}, e.g. $NL_{\phi
    \rightarrow G}$ for the nonlinearity involving $\phi$ in the evolution
  equation for $G$, for ease of comparison.

  \underline{The quadratic nonlinearity in $\phi$:}
  As discussed above, a key role is played by the nonlinear self-interaction of
  $\phi$ and the (lack of) control of
  \begin{align*}
   NL_{\phi}&:= \langle A\big((\p_x\nabla^\perp \Delta_t^{-1} \phi \cdot \nabla)\phi\big) -(\p_x\nabla^\perp \Delta_t^{-1} \phi_{\neq}\cdot \nabla) A\phi ,A\phi\rangle,
  \end{align*}
  where we used fact that $\phi\mapsto \nabla^{\perp}a \cdot \nabla \phi$ is an
  antisymmetric operator for any smooth function $a$.
  By using Plancherel's Theorem we may equivalently express $NL_{\phi}$ in
  Fourier space as
  \begin{align*}
    NL_{\phi}&=i\sum_{k,l} \iint d\xi \ d\eta \ A\phi(t,k,\eta) A\phi(t,l,\xi) A\overline{\phi}(t,k-l,\eta-\xi) \\
    & \qquad \times \frac{A(t,k,\eta)- A(t,k-l,\eta-\xi)}{A(t,l,\xi)A(t,k-l,\eta-\xi)} \frac{l (\xi(k-l)-l(\eta-\xi))}{l^2+(\xi-lt)^2}.
  \end{align*}
  and claim that
  \begin{align*}
    |NL_{\phi}| \lesssim \langle t\rangle \|A\phi\|^3 \leq \langle t \rangle E(t)^{3/2}.
  \end{align*}
  Furthermore, we note that for $l=k-1$, $\eta=\xi$, $\xi-lt=0$, it holds that
  \begin{align*}
     \frac{l (\xi(k-l)-l(\eta-\xi))}{l^2+(\xi-lt)^2} = t
  \end{align*}
  and thus we cannot expect to obtain a better bound.

  In order to estimate $NL_{\phi}$ we employ a paraproduct decomposition. More
  precisely, since $(l,\xi) + (k-l, \eta-\xi)=(k,\xi)$ at least one of these vectors has to
  be of comparable or larger magnitude than $(k,\xi)$.
  Here, in the literature one commonly refers to the \emph{transport} regime
  \begin{align*}
    \langle l, \xi  \rangle \ll \langle k, \eta \rangle \text{ and } \langle k-l, \eta- \xi  \rangle \approx \langle k, \eta \rangle,
  \end{align*}
  the \emph{reaction} regime
   \begin{align*}
    \langle l, \xi  \rangle \approx \langle k, \eta \rangle \text{ and } \langle k-l, \eta- \xi  \rangle \ll \langle k, \eta \rangle.
   \end{align*}
   and the \emph{remainder} regime
   \begin{align*}
     \langle l, \xi  \rangle \gtrsim \langle k, \eta \rangle \text{ and } \langle k-l, \eta- \xi  \rangle \gtrsim \langle k, \eta \rangle,
   \end{align*}
   and correspondingly we split the integral as
   \begin{align*}
     NL_{\phi} = T_{\phi} + R_{\phi} + \mathcal{R}_{\phi}.
   \end{align*}

   \underline{Estimating the remainder $\mathcal{R}_{\phi}$:}
   In the remainder region, by the triangle inequality it holds that $\langle l,
   \xi  \rangle \approx \langle k-l, \eta- \xi  \rangle \gtrsim \langle k, \eta
   \rangle$ and therefore
   \begin{align*}
     \left| \frac{A(t,k,\eta)- A(t,k-l,\eta-\xi)}{A(t,l,\xi)A(t,k-l,\eta-\xi)}\right| \lesssim \langle l, \xi  \rangle^{-N}.
   \end{align*}
   Since $N\geq 5$ is much larger than the Sobolev embedding threshold, we can further
   easily absorb $\frac{l (\xi(k-l)-l(\eta-\xi))}{l^2+(\xi-lt)^2}$ and hence
   infer that
   \begin{align*}
     |\mathcal{R}_{\phi}| \lesssim \|A\phi\|^3 \leq E^{3/2}.
   \end{align*}

   \underline{Estimating the reaction $R_{\phi}$:}
   We estimate
   \begin{align*}
     \frac{|l (\xi(k-l)-l(\eta-\xi))|}{l^2+(\xi-lt)^2} \leq \langle t \rangle \langle k-l, \eta-\xi \rangle.
   \end{align*}
   If $\langle k-l, \eta- \xi  \rangle \gtrsim \langle k, \eta
   \rangle^{\frac{3}{N}}$ we may argue as in the remainder regime. Otherwise, we
   can absorb the factor $\langle k-l, \eta-\xi \rangle$ into the gain by the
   Sobolev embedding. We thus conclude that
   \begin{align*}
     |R_{\phi}| \lesssim \langle t \rangle \|A\phi\|^3 \leq \langle t \rangle E^{3/2}.
   \end{align*}

   \underline{Estimating the transport $T_{\phi}$:}
   By the same argument as above, we may assume that $\langle l, \xi \rangle
   \lesssim \langle k, \eta \rangle^{\frac{3}{N}}$.
   Since $\langle k-l, \eta-\xi \rangle \approx \langle k, \eta \rangle$ might
   be very large, we here rely on the gain by the commutator $A(t,k,\eta)-
   A(t,k-l,\eta-\xi)$ to control
   \begin{align*}
     \frac{|l (\xi(k-l)-l(\eta-\xi))|}{l^2+(\xi-lt)^2} \leq \langle l, \xi \rangle \langle k-l, \eta-\xi \rangle.
   \end{align*}
   More precisely, solving explicitly \eqref{def:mL} we find that
   \begin{align*}
     A(t,k,\eta)&= \langle k,\eta\rangle^N\exp\left(\int_0^t -\frac{5}{1+(\tau-\frac{\eta}{k})^2} 1_{|k,\eta|\leq 10 \langle \tau \rangle} d\tau\right)\\
     &=: \langle k,\eta\rangle^N m_{L}^{-1}(t,k,\eta).
   \end{align*}
   We thus split the commutator of $A$ into two distinct commutators:
   \begin{align*}
     & \qquad A(t,k,\eta)- A(t,k-l, \eta-\xi) \\
     &= \left(\langle k,\eta\rangle^N- \langle k-l,\eta-\xi\rangle^N\right) m_{L}^{-1}(t,k,\eta)\\
     & \qquad + \langle k-l,\eta-\xi\rangle^N\left(m_{L}^{-1}(t,k,\eta) - m_{L}^{-1}(t,k-l,\eta-\xi)  \right).
   \end{align*}
   Clearly, the first commutator gains a factor $\langle k,\eta \rangle^{-1}$
   and thus allows for an estimate of $|T_{\phi}|$ by $\|A\phi\|^3$.
   
   For the second commutator we distinguish between two cases:

   \begin{itemize}
   \item 
   If $|k,\eta|\geq 10 \langle t \rangle$, then
   $m_L(t,k,\eta)=m_L(t,k-l,\eta-\xi)=1$ and hence the commutator is trivial.
 \item  If instead $|k, \eta|\leq 10 \langle t \rangle$, then the loss of
   derivative is controlled by a loss of a factor $t$. We may thus bound
   \begin{align*}
     |T_{\phi}|\lesssim \langle t \rangle \|A\phi\|^{3}.
   \end{align*}
   \end{itemize}

   \underline{The action of $G$ on $\phi$:}
   The argument for the action of $G$ on $\phi$ is largely analogous to the
   previous case. In order to avoid duplication with
   Section \ref{sec:bootstrap}, where we establish more optimized bounds for
   more intricate weights, we discuss the different regimes briefly.
   
   Using Plancherel's identity, we consider
  \begin{align*}
    NL_{G \rightarrow \phi}&:= \langle A \phi, A(\nabla^{\perp}G \cdot \nabla \phi) \rangle = \langle A \phi, A(\nabla^{\perp}G \cdot \nabla \phi) -\nabla^{\perp}G \cdot \nabla A \phi  \rangle\\
    &= \sum_{k,l} \int d\eta d\xi A\phi(t,k,\eta) AG(t,l,\xi) A\overline{\phi}(t,k-l,\eta-\xi) \\
    & \qquad \times \frac{A(t,k,\eta)-A(t,k-l,\eta-\xi)}{A(t,l,\xi)A(t,k-l,\eta-\xi)} (\xi(k-l)-l(\eta-\xi)).
  \end{align*}

  In the remainder case, we obtain a straightforward bound by $E^{3/2}$.
  In the reaction case, we control
  \begin{align*}
    \|A \nabla^{\perp}G\| \leq \langle t \rangle \|A\nabla_t G\|
  \end{align*}
  and recall that $\|A\nabla_t G\|$ is square-integrable in time due to the
  resistive dissipation.
  Finally, in the transport regime we again exploit the commutation properties
  of $A$, as in the case above.
  Thus, over all we obtain a bound
  \begin{align*}
    NL_{G \rightarrow \phi} \leq \langle t \rangle E^{3/2} + \langle t \rangle E \|A\nabla_t G\|.
  \end{align*}

  \underline{The remaining nonlinear contributions:}
  The estimates for the remaining nonlinearities
  \begin{align*}
    NL_{\phi \to G}:=  \langle A\big(\Delta^{-1}_t \nabla^\perp_t \cdot(( \nabla^\perp \phi \cdot   \nabla) \nabla^\perp_t  \phi )_{\neq}\big), AG\rangle, \\
    NL_G:=-\langle A\big( \Delta^{-1}_t  ( v\cdot \nabla_t w)_{\neq} +  \p_x\Delta_t^{-1}(v\cdot\nabla_t  \phi)_{\neq}\big),AG\rangle,
  \end{align*}
  can be shown by the same arguments and are comparatively simpler in view of the
  additional dissipation of $G$.
  We thus obtain rough bounds
  \begin{align*}
    |NL_{\phi \to G}| + |NL_{G}| \leq \langle t \rangle E^{3/2} + \langle t \rangle E \|A\nabla_t G\|.
  \end{align*}
  These estimates are revisited in Proposition \ref{prop:NLbounds} (with
  improved weights) and thus a detailed discussion is omitted here for brevity.
  
\end{proof}
The perturbative time scale is thus determined by a standard energy argument and
the possible linear growth bound on the magnetic field (this growth is attained,
see Theorem \ref{thm:main}). In the following, we show that this estimate is
optimal and that nonlinear effects become dominant on longer time scales. Here,
in a first step we discuss nonlinear resonances in a model problem and the role
of viscous dissipation and magnetic coupling.

\subsection{Nonlinear Resonances and Paraproducts}
\label{sec:resonance}

While the linearized dynamics provide a good approximation for small data and on
corresponding perturbative time scales, this generally ceases to be the case on
longer time scales or for larger data. Indeed, in case of algebraic growth in
the linearized dynamics, this approximation may even fail to be consistent on
longer time scales, such as in Backus' objection to Landau damping
\cite{backus1960linearized}. It thus becomes crucial to capture nonlinear
effects and to identify for which frequencies and at which times large
corrections, so called resonances, can occur.

In this section, we argue on a heuristic level and in terms of paraproducts
(i.e., frequency decompositions into ``high'' and ``low'' frequencies) to
introduce the main nonlinear coupling and growth mechanisms. In particular, we
show that these possibly large corrections to the linearized dynamics are highly
frequency- and time-localized. We thus speak of \emph{resonances}.

We recall from \eqref{eq:phiGeq} that the nonlinear equations (\ref{eq:NL}) in
terms of the unknowns $G$ and $\phi$ read as
\begin{align*}
  \p_t G    & = (\Delta_t+2  \p_x \p_y^t\Delta_t^{-1} +  \p_x^2 \Delta_t^{-1}) G  - \p_x^3 \Delta_t^{-2} \phi                                      \\
            &\quad + \Delta^{-1}_t \p_x \nabla^\perp \phi \nabla \Delta_t \phi \\
            &\quad -\Delta^{-1}_t  ( v\cdot \nabla_t w)-   \p_x\Delta_t^{-1}(v\cdot\nabla_t  \phi),\\
  \p_t \phi & =- \p_x^2 \Delta_t^{-1} \phi +  \p_x  G \\
            &\quad -     \nabla^\perp G\nabla \phi - (\p_x\nabla^\perp \Delta_t^{-1} \phi_{\neq}\cdot \nabla) \phi.
\end{align*}
The main challenges in the following are thus posed by the control of the
quadratic nonlinearities. In our analysis, we consider these
nonlinearities as forcing terms acting on the linear evolution. However, while
parts of these forces can be treated perturbatively, some contributions turn out
to be very large and qualitatively change the dynamics. In the following
subsection we introduce these effects on a heuristic level, which serves as a
starting point for the rigorous analysis of Section \ref{sec:bootstrap}.

\subsection{Main Resonance Mechanisms}
In Section \ref{sec:linear} we have shown that the linearized evolution
equations are stable in Sobolev regularity. Furthermore, as shown in Lemma
\ref{lem:perturbative} this stability extends to the nonlinear equations for
small data on perturbative time scales. However, on longer time scales it can a
priori not be expected that the nonlinear dynamics remain close to the linear
ones. In this section, on a heuristic level, we show that this expectation
strongly fails and that the quadratic nonlinearities result in drastically
different stability properties.

In the interest of a clear and tractable heuristic model, we make several
simplifications to the model (see Section \ref{sec:bootstrap} for the full
model):
\begin{itemize}
\item Since the linear dynamics turn out to be stabilizing, we omit all linear
  terms except for the dissipation in $G$ and part of the coupling by the
  underlying magnetic field (which could propagate some of this dissipation to
  $\phi$).
\item We are mainly interested in the behavior at high frequencies, since there
  different stability properties and regularities are most visible. We thus
  focus on the evolution ``at high frequency'' and split all quadratic terms
  into ``high'' and ``low'' frequency parts, e.g.,
  \begin{align*}
    \nabla^{\perp}G \nabla \phi \leadsto \nabla^{\perp}G_{hi} \nabla \phi_{lo} + \nabla^{\perp}G_{lo} \nabla \phi_{hi} + \nabla^{\perp}G_{hi} \nabla \phi_{hi}.
  \end{align*}
  Note that there is no low-low part, as we are interested in the high frequency
  evolution. These splittings can be made precise in terms of Littlewood-Payley
  projections and paraproducts, but for the purposes of this section we think of
  low frequencies as of size about $1$ and high frequency about $\lambda$ with
  $\lambda \gg 1$ very large.
\item This further implies that derivatives are comparable to multiplication
  with $\lambda$, that is, $\dt D^N \phi_{hi}\approx \lambda^{N}\dt \phi_{hi}$.
  In particular, it follows that derivatives are only relevant for high
  frequency contributions and
  \begin{align*}
    \lambda^N\nabla^{\perp}G_{hi} \nabla
    \phi_{hi} \approx \nabla^{\perp} \lambda^{N/2} G_{hi} \nabla \lambda^{N/2} \phi_{hi}.
  \end{align*}
  For this reason, all high-high terms can be expected to be rather harmless.
\item Similarly, we assume that the $H^N$ norm of all low frequency parts
  remains bounded by a small constant $\epsilon>0$ uniformly in time (but does
  not necessarily decay). In particular, we may heuristically replace
  $\Delta_{t}^{-1}\phi_{lo}\approx \frac{\epsilon}{1+t^2}$.
\item In view of these simplifications, the most dangerous nonlinearities are
  those where additional derivatives are applied to the high frequency parts,
  since those may a priori be of size $\lambda\gg 1$.
\end{itemize}

With these preparations, we isolate the following contributions:
\begin{align*}
  \begin{split}
    \p_t G -\Delta_t G&=  \Delta^{-1}_t \p_y^t (( \nabla^\perp \phi^{hi}  \nabla) \p_y^t \phi^{lo} )_{\neq}, \\
    \p_t \phi_{\neq}  &= \p_x  G+  ( \nabla^\perp G^{hi}\nabla \phi^{lo})_{\neq}.
  \end{split}
\end{align*}
Here, the unknowns $G, \phi_{\neq}$ correspond to high frequency perturbations
with non-zero horizontal frequency. Focusing on the evolution at a given
frequency $(k,\eta)\in (\Z\setminus \{0\}) \times \R$, if $|\eta-kt|\gg 1$ is
large, the operator $\Delta^{-1}_t \p_y^t\approx |\eta-kt|^{-1}$ provides
smallness and time decay.

Thus, we focus on times $t$ such that $t\approx \tfrac \eta k$ and, as a worst
case, consider $\phi^{lo}\approx \epsilon \cos(x)$. In this way, the system in
Fourier space becomes a system of coupled ordinary differential equations with
nearest neighbor coupling. Moreover, unlike the full nonlinear system, it
decouples with respect to the vertical frequency $\eta$. We thus introduce a
\underline{three-mode model}:
\begin{align}
  \label{eq:modelproblem}
  \begin{split}
    \p_t G(k) + k^2(1+(t-\tfrac \xi k )^2 ) G(k) &=\eps t \tfrac \eta k \tfrac 1 {1+\vert t-\frac \eta k \vert }\phi(k+1),\\
    \p_t \phi(k+1) &=\eps t k  G(k), \\
    \p_t \phi(k) &= ik G(k),
  \end{split}
\end{align}
where we omit all other modes and consider times $t\approx \tfrac \eta k$. In
particular, we note that for $t=\frac{\eta}{k}$, the right-hand-side in the
first equation equals $\epsilon t^2 \phi(k+1)$ while the dissipation is only of
size $k^2G(k)$. Thus, for times $t\gtrsim \epsilon^{-1/2}$ past the perturbative
time scale (see Section \ref{sec:perturbative}), the good unknown $G$ cannot be
expected to remain small even if it was initially zero. Instead, neglecting the
time-dependence of $\phi(k+1)$, we obtain that
\begin{align*}
  G(k)\approx \eps t \tfrac \eta{k^3} \tfrac 1{\langle t-\frac \eta k \rangle^3 }\phi(k+1).
\end{align*}
Inserting this into our model yields a coupled system of two modes:
\begin{align*}
  \p_t \phi(k+1)&= (\eps t)^2 \tfrac \eta{k^2} \tfrac 1{\langle t-\frac \eta k \rangle^3 }\phi(k+1),\\
  \p_t \phi(k) &=  i\eps t \tfrac \eta{k^2} \tfrac 1{\langle t-\frac \eta k \rangle^3 }\phi(k+1).
\end{align*}
If $ t \le \delta \eps^{-\frac23}$ we obtain
\begin{align*}
  \phi(k+1) &\approx 1,\\
  \phi(k) &\approx \tfrac {\eta^{\frac 1 2 }}{k^{\frac 32 }} \phi(k+1).
\end{align*}
In particular, iterating this for a sequence of frequencies $k$ with $\eta$ fixed
(similar to echo chains in Landau damping \cite{bedrossian2020nonlinear}), this
suggests that
\begin{align*}
  \phi(1) \approx \phi(K) \prod_{k=1}^{K} \tfrac {\eta^{\frac 1 2 }}{k^{\frac 32 }} \underset{K\sim \sqrt[3]{\eta}}{\approx} \phi(K) \exp(\tfrac{1}{2}|\eta|^{1/3}).
\end{align*}
Such growth in Fourier space exactly corresponds to Gevrey 3 regularity.

We remark that on time scales much larger than this, i.e., $\eps t^{\frac 32 }
\gg 1 $, we instead obtain highly degenerate growth. More precisely, for $\tfrac
\eta {k^2} \ge 1 $ and $\phi_{in} (k+1)=1$, we deduce that
\begin{align*}
  \phi(k+1) &\approx \exp\left((\eps t)^2 \tfrac \eta{k^2} \int \tfrac 1{\langle \tau-\frac \eta k \rangle^3 } d\tau  \right)\\
  \phi(k) &\approx \tfrac 1 {\eps t }(\exp\left((\eps t)^2 \tfrac \eta{k^2} \int \tfrac 1{\langle \tau-\frac \eta k \rangle^3 } d\tau  \right)-1)\\
            &\gtrsim \eps t \tfrac \eta{k^2} \int \tfrac 1{\langle \tau-\frac \eta k \rangle^3} d\tau.
\end{align*}
This exponential growth, present even in the case $\eps^{-\frac23}\lesssim
t\lesssim \eps^{-1}$, cannot anymore be controlled by Gevrey regularity (similar
effects are expected for the inviscid Boussinesq beyond the natural time scale
considered in \cite{bedrossian2023nonlinear}).

\section{Nonlinear Stability: Energy Functionals and Bootstrap Approach}
\label{sec:bootstrap}
Building on the insights of Section \ref{sec:resonance}, our main aim in the
following is to establish nonlinear stability of the non-resistive MHD equations
\eqref{eq:NL} in terms of the unknowns $G$ and $\Phi$.
In particular, the resonances of the model problem \eqref{eq:modelproblem} suggest
that any stability result will require Gevrey regularity with some loss of the
radius of regularity in time.
We recall the definition of $G$ in \eqref{def:G} and the full nonlinear system we aim at studying is
\begin{align}
    \label{eq:Gsec4}\p_t G &= (\Delta_t+ 2 \p_x \p_y^t\Delta_t^{-1} + \p_x^2 \Delta_t^{-1}) G- \p_x^3 \Delta_t^{-2} \phi_{\neq}\\
    \notag &\quad + \Delta^{-1}_t \nabla^\perp_t \cdot(( \nabla^\perp \phi  \cdot\nabla) \nabla^\perp_t  \phi )_{\neq}-\Delta^{-1}_t  ( v\cdot \nabla_t w)_{\neq} -   \p_x\Delta_t^{-1}(v\cdot\nabla_t  \phi)_{\neq}, \\
    \label{eq:Phisec4}\p_t \phi   &=- \p_x^2 \Delta_t^{-1} \phi_{\neq} + \p_x G \\
  \notag  &\quad - ( \nabla^\perp G\cdot \nabla \phi) + (\p_x\nabla^\perp \Delta_t^{-1} \phi_{\neq}\cdot \nabla) \phi - v^x_{0}\p_x \phi_{\neq},  \\
   \label{eq:v0sec4} \p_t v^x_{0}&= \p_y^2 v^x_{0} +( b_{\neq} \cdot\nabla_t b_{\neq}^x -v_{\neq} \cdot\nabla_t v^x_{\neq})_0.
\end{align}
Our goal is to perform weighted energy estimates on $G,\phi,v_0^x$. To this end, we introduce the Fourier multiplier
\begin{align}
    \label{def:A0}
    A(t,k,\eta)&= \langle k,\eta  \rangle^N  (m^{-1}  J)(t,k,\eta ) e^{\lambda(t) \vert \eta ,k \vert^s},
\end{align}
where we define $m,\,  J$  in Section \ref{sec:weights}. We also need the technical variant $\tilde{A}$ which is defined as above with $J\to \tilde{J}$ that is defined in \eqref{def:tJ}.  At this stage of the discussion, it is enough to know that there exists $0<\tilde \lambda\leq \lambda(t)$ such that
\begin{align*}
\|f\|_{\mathcal{G}^{\tilde \lambda}}\lesssim \|A f\|_{L^2}.
\end{align*}

Moreover, the weight $m$ is introduced to handle  terms that are heuristically integrable in time. Instead, the weight $J$ takes into account the delicate anisotropic regularity loss and in particular is based on the weight introduced in \cite{masmoudi2022stability}, which is well adapted to the Gevrey-3 regularity loss we encounter in our problem.

We define our main energy functionals as follows:
\begin{align}
\label{def:ENL}
\mathcal{E}(t)&:=\frac12 \left(\|AG\|_{L^2}^2+\|A\phi\|_{L^2}^2\right),\\
\label{def:E0}\mathcal{E}_0(t)&:=\frac12\| A \langle \p_y\rangle^{-1}v_0^x\|_{L^2}^2.
\end{align}
\begin{remark}[On the zero-mode energy] Observe that we need to control the $x$-average of the velocity with one Sobolev derivative less. At high frequencies, this is equivalent at controlling $\psi_0=-\p_y^{-1}v_0^x$, which would be the natural quantity to handle. Indeed, all our unknowns are defined via the streamfunctions $\phi,\psi$. However, at low frequency there are some technical issues that could be easily handled by estimating $v_0^x$. The latter is expected to decay faster than $\psi_0^x$ by standard heat-semigroup bounds.
\end{remark}
Our goal is to control the energies defined in \eqref{def:ENL}-\eqref{def:E0} and prove that they remain of size $O(\eps^2)$ up to a time-scale $O(\eps^{-\frac23})$.  Indeed, as we show below, this would readily imply Theorem \ref{thm:main}. To prove the bounds on the energies, we proceed with a bootstrap argument. We first define the associated \textit{dissipation functionals} arising when taking the time-derivative of our energies as
    \begin{align}
    \label{def:DGphi}
    \mathcal{D}_{G,\phi}:&=
    \|A \nabla_tG\|_{L^2}^2+\|\p_x \Lambda_t^{-1} \phi_{\neq}\|^2_{L^2}+\sum_{\iota \in \{\lambda,m,q\}} \mathcal{D}_{\iota}[G]+\mathcal{D}_{\iota}[\phi_{\neq}],\\
    \label{def:Dv0}\mathcal{D}_{v_0}:&= \| A\langle \p_y \rangle^{-1} \p_yv_0^x \|_{L^2}+ \mathcal{D}_{\lambda}[\langle \p_y \rangle^{-1}v_{0}^x].
    \end{align}
Then, inspired by the bounds available for the linearized system, we formulate our \textit{bootstrap hypotheses}: for two given constants $c_*,C_*>1$ and a time $0<T\leq\delta \eps^{-2/3}$, we assume that
\begin{align}\tag{Bh}\begin{split}
     \label{bootstrap}
     & \quad \mathcal{E}(0)+\mathcal{E}_0(0)\leq c_*\eps^2,\\
    &\sup_{t\in [0,T]}\mathcal{E}(t)+\frac{1}{100}\int_{0}^{T}\mathcal{D}_{G,\phi}(\tau) \mathrm{d} \tau\leq 4C_* \eps^2,\\
    &\sup_{t\in [0,T]}\mathcal{E}_0(t)+\frac{1}{100}\int_{0}^{T}\mathcal{D}_{v_0}(\tau) \mathrm{d} \tau\leq 4C_* \eps^2.
\end{split}\end{align}
By standard local well-posedness theory (see also Section \ref{sec:perturbative}), we know that \eqref{bootstrap} are verified with $4C_*$ replaced by $C_*$ if one takes $T$ sufficiently small or $T=1$ and $\eps$ sufficiently small, the latter being the case we will consider from now on. This implies that $T\geq 1$ and therefore our goal is now to prove that \eqref{bootstrap} holds true with $4C_*$ replaced by $2C_*$ on the whole time-interval $(0,T)$, meaning that $T=\delta \eps^{-\frac23}$ as desired.

To accomplish this task, we first compute the energies identities which are the starting point for the nonlinear analysis.
\begin{lemma}
\label{lem:energy}
The following energy identities holds true
    \begin{align}
    \label{eq:dtE}
        \frac{\mathrm{d}}{\mathrm{d} t}\mathcal{E}+&\|A \nabla_tG\|_{L^2}^2+\|\p_x \Lambda_t^{-1} A\phi_{\neq}\|^2_{L^2}+\sum_{\iota \in \{\lambda,m,q\}} \mathcal{D}_{\iota}[G]+\mathcal{D}_{\iota}[\phi_{\neq}]\\
    \notag&=L_{G,\phi}+NL_{\phi \to G}+NL_{G \to \phi}+NL_\phi+NL_G
    \end{align}
    and
 \begin{align}
        \frac{\mathrm{d}}{\mathrm{d} t}\mathcal{E}_0&+\| A \langle \p_y\rangle^{-1}\p_y v_0^x\|_{L^2}^2+ \mathcal{D}_{\lambda}[\langle \p_y\rangle^{-1}v_0^x]= NL_{v_0}
    \end{align}
    where we define:
    \begin{itemize}
        \item The artificial damping terms are
        \begin{align}
\label{def:Dm}\mathcal{D}_m[f]&:=\big\|\sqrt{\frac{\p_tm}{m}}Af\big\|^2_{L^2},\\
\label{def:Dq}\mathcal{D}_q[f]&:=\big\|\sqrt{\frac{\p_t q}{q}}\tilde Af\big\|^2_{L^2},\\
\mathcal{D}_\lambda [f]&:=\big\| \vert \sqrt{- \p_t \lambda} \vert\nabla\vert^{\frac s 2 }  Af\big\|^2_{L^2}.
        \end{align}
        \item The linear term is
        \begin{align}
\label{def:LGphi}
L_{G,\phi}:=\, &\langle (2 \p_x \p_y^t\Delta_t^{-1} +  \p_x^2 \Delta_t^{-1}) AG, AG\rangle \\
           \notag &-\langle \big(\p_x^3 \Delta_t^{-2}+\p_x\big)A\phi_{\neq} , AG\rangle.
        \end{align}
        \item The nonlinear action of $\phi$ in $G$ is
        \begin{equation}
            \label{def:NLphitoG}
            NL_{\phi \to G}:=  \langle A\big(\Delta^{-1}_t \nabla^\perp_t \cdot(( \nabla^\perp \phi \cdot   \nabla) \nabla^\perp_t  \phi )_{\neq}\big), AG\rangle.
        \end{equation}
        \item The nonlinear action of $G$ in $\phi$ is
        \begin{equation}
            \label{def:NLGtophi}
NL_{G \to \phi}:= -\langle A( \nabla^\perp G\cdot\nabla \phi) , A\phi\rangle.
        \end{equation}
        \item  The remaining nonlinear terms arising from the equation of $\phi$ are
        \begin{equation}
    NL_{\phi}:= \langle A\big((\p_x\nabla^\perp \Delta_t^{-1} \phi_{\neq}\cdot \nabla)\phi\big),A\phi\rangle -\langle A(v^x_{0}\p_x \phi_{\neq}),A\phi_{\neq}\rangle.
        \end{equation}
        \item The remaining nonlinear terms arising from the equation of $G$ are
        \begin{equation}
            \label{def:NLG}
            NL_G:=-\langle A\big( \Delta^{-1}_t  ( v\cdot \nabla_t w)_{\neq} +  \p_x\Delta_t^{-1}(v\cdot\nabla_t  \phi)_{\neq}\big),AG\rangle.
        \end{equation}
        \item The nonlinear action on $v_0^x$ is
\begin{equation}
\label{def:NLv0}
            NL_{v_0}:=  \langle A\langle \p_y\rangle^{-1}( b_{\neq} \cdot\nabla_t b_{\neq}^x -v_{\neq}\cdot \nabla_t v^x_{\neq})_0, A\langle \p_y\rangle^{-1}v_0^x\rangle .
        \end{equation}
    \end{itemize}
\end{lemma}
\begin{proof}
    The energy identity \eqref{eq:dtE} for $\mathcal{E}$ is a direct computation. Indeed, the first two good terms on the left hand side of \eqref{eq:dtE} are easily obtained from the first terms on the right hand side of \eqref{eq:Gsec4} and \eqref{eq:Phisec4}.  The artificial damping terms arise from the time derivative of the weight $A$. The remaining linear terms are obtained from the first line on the right hand side of \eqref{eq:Gsec4} and \eqref{eq:Phisec4}, where we just performed an integration by parts for the term $\p_x G$ in the equation \eqref{eq:Phisec4}. All the nonlinear terms are straightforward.
\end{proof}
The goal of the rest of the paper is to control the terms on the right hand side of our energy identity. From now on, without loss of generality we will always consider $t\geq 1$ (since otherwise it is enough to replace $t$ with $\langle t \rangle$).

First of all, by the results of Section \ref{sec:linear} we have the following bound on the linear terms.
\begin{lemma}
\label{lem:LGphi}
    Let $L_{G,\phi}$ be defined as in \eqref{def:LGphi}. Then,
  \begin{equation}
        \label{bd:LGphi}
        |L_{G,\phi}|\leq \frac{1}{2}\left(\mathcal{D}_{m}[G] + \mathcal{D}_{m}[\phi] + \|A \nabla_tG\|_{L^2}^2+\|\p_x \Lambda_t^{-1} A\phi_{\neq}\|^2_{L^2}\right).
    \end{equation}
\end{lemma}
In fact, the bound on the linear terms essentially follow by the analysis of Section \ref{sec:linear}, but we present some details for clarity.
\begin{proof}[Proof of Lemma \ref{lem:LGphi}]
We apply the results of Proposition \ref{prop:linear}. More precisely, we note that the weight $m^{-1}$ satisfies the bound \eqref{eq:Alin} and hence we obtain the energy estimate \eqref{eq:linearestimate}. This, in particular, yields the desired bound \eqref{bd:LGphi}. 

\end{proof}

Then, in order to improve \eqref{bootstrap}, we need to establish improved bounds on the various nonlinear action terms, which we collect in the next key proposition.
\begin{prop}
\label{prop:NLbounds}
    Consider the energy identity \eqref{eq:dtE}. There exists constants $C_i>0$, $i=1,\dots,5$ for which, under the bootstrap hypothesis \eqref{bootstrap}, it holds that
\begin{align}
    \label{bd:NLphiG}\int_0^T |NL_{\phi \to G}| dt&\leq C_1\delta \eps^2,\\
     \label{bd:NLGphi}\int_0^T |NL_{G \to \phi}|dt&\leq C_2\delta \eps^2,\\
\label{bd:NLG}\int_0^T |NL_{G}|dt&\leq C_3\delta \eps^2,\\
\label{bd:NLphi}\int_0^T |NL_{\phi}|dt&\leq C_4\delta \eps^2,\\
\label{bd:NLv0}\int_0^T |NL_{v_0}|dt&\leq C_5\delta \eps^2.
\end{align}
\end{prop}
Finally, with the proposition above and Lemma \ref{lem:energy} at hand, we are ready to show the proof of the stability part of the main Theorem \ref{thm:main}.
\begin{proof}[Proof of stability in Theorem \ref{thm:main}]
    Let $0<T\leq \delta \epsilon^{-2/3}$ be the maximal time such that the bootstrap hypothesis \eqref{bootstrap} holds.
    Then integrating the energy inequality of Lemma \ref{lem:energy} and using the bounds of Proposition \ref{prop:NLbounds}, it follows that the energy satisfies an improved bootstrap hypothesis
    \begin{align}
    \tag{Bh} \label{bootstrap2}
    \begin{split}
    &\sup_{t\in [0,T]}\mathcal{E}(t)+\frac{1}{100}\int_{0}^{T}\mathcal{D}_{G,\phi}(\tau) \mathrm{d} \tau\leq C_* \epsilon^2,
    \end{split}
    \end{align}
    whenever $\delta\sum_{i=1}^5C_i\leq C_*.$
    If $T< \delta \eps^{-\frac{2}{3}}$ this improved bound combined with local in time wellposedness implies that the (non-improved) bootstrap hypothesis \eqref{bootstrap} holds at least for a short additional time, which contradicts the maximality of $T$.
    Therefore, the estimate \eqref{bootstrap} is satisfied with $T=\delta \eps^{-\frac{2}{3}}$.

    Having established this uniform bound on the energy functional, the upper bounds in Theorem \ref{thm:main}   are a straightforward consequence of the properties of $A$.
\end{proof}
The instability in Theorem \ref{thm:main} is also a consequence of the fact that the bootstrap estimates \eqref{bootstrap} holds true on the time interval $[0,\delta \eps^{-\frac23}]$. Since this requires some technical effort, we postpone the proof of the instability to Section \ref{sec:Instability} and below we present the proof of Proposition \ref{prop:NLbounds}.

\subsection{Weights and Frequency Sets}
\label{sec:weights}
In this section, we introduce the main technical tools needed in the rest of the paper.

\subsubsection{Useful sets}
We first define some time-frequency sets, which are useful to split the main nonlinear interactions and define the weights.

The \textit{reaction, transport and remainder sets} are
\begin{align}
    \label{def:SR}S_R &=\{ ((k,\eta),(l ,\xi )) \in (\mathbb{Z}\times \mathbb{R})^2\, :\, \ \vert k-l,\eta -\xi \vert\ge 8 \vert l, \xi\vert\}, \\
    \label{def:ST}S_T &=\{ ((k,\eta),(l ,\xi )) \in (\mathbb{Z}\times \mathbb{R})^2\, :\, \ 8\vert k-l,\eta -\xi \vert\le  \vert l, \xi\vert \},  \\
    \label{def:ScalR}S_\calR  &=\{ ((k,\eta),(l ,\xi )) \in (\mathbb{Z}\times \mathbb{R})^2 \, :\, \ \tfrac 1 8  \vert l, \xi\vert\le \vert k-l,\eta -\xi \vert\le  8 \vert l, \xi\vert\} .
\end{align}
These sets are introduced to split the nonlinear terms by a standard paraproduct decomposition (see for instance \cite{masmoudi2022stability}).

The following set
 \begin{align}
    S_{t} &= \{(k,\eta)\in \mathbb{Z}\times \mathbb{R}\, :  \, \vert k,\eta \vert\leq 10 t^{2} \},
\end{align}
is used to apply an adapted weight if the time exceeds a high frequency threshold.

As is standard in the field by now, we introduce the resonant times and intervals, where in particular we follow the definitions in  \cite{masmoudi2022stability}. For $1\le \vert k\vert \le \vert \eta\vert ^{\frac 1 3 }   $  we define
\begin{align*}
    t_{k,\eta}^\pm&= \frac \eta k \pm \frac \eta{2k^3 } \\
    \tilde I_{k,\eta } &= \begin{cases}
    [t^-_{k,\eta},t^+_{k,\eta}] &\qquad 1\le \vert k \vert \le \vert \eta \vert^{\frac 1 3}, \ \eta k >0,\\
        \emptyset & \qquad\text{otherwise} .
    \end{cases}   \\
    I_{k,\eta } &= \begin{cases}
    [\tfrac 1 2 (\tfrac \eta k +\tfrac {\vert \eta\vert } {\vert k\vert +1}),\tfrac 1 2 (\tfrac \eta k +\tfrac {\vert \eta\vert } {\vert k\vert -1})] &\qquad 2\le \vert k \vert \le \vert \eta \vert^{\frac 1 3}, \ \eta k >0\\
     [\tfrac \eta  3 ,2\eta]&  \qquad k = 1, \\
        \emptyset & \qquad \text{otherwise}.
    \end{cases}
\end{align*}
We also denote
\[
\tilde I_{k,\eta}=\tilde I_{k,\eta}^L\cup \tilde I_{k,\eta}^R, \qquad \tilde I_{k,\eta}^L=[t_{k,\eta}^{-},\tfrac{\eta}{k}],
\qquad \tilde I_{k,\eta}^R=[\tfrac{\eta}{k},t_{k,\eta}^{+}].
\]
The same convention is used for $I_{k,\eta}$ with $k\geq 2$ when we replace $t_{k,\eta}^{\pm}$ with $\eta/(2k)+\eta/(2(k\mp 1)$. For $I_{1,\eta}$ instead $I_{1,\eta}^L=[\eta/3,\eta]$ and $I_{1,\eta}^R=[\eta,2\eta]$.
\subsubsection{Definition of weights}
We are now in the position of defining the weights. Recall the definitions of the weights $A$ and $\tilde A$ in \eqref{def:A0}. Here we use the notation for the sets introduced in the previous section.

First of all, to define the function $\lambda:\mathbb{R}\to \mathbb{R}$ we fix $\rho_0 \in (0,1)$, $\gamma_*\in (0,\frac32(s-\frac13))$ and set
 \begin{align}
 \label{def:lambda}
    \begin{cases}
            \p_t\lambda &=-  \rho_0\langle t\rangle^{-(1 +\gamma_*) },\\
            \lambda(0)&= \lambda_0,
    \end{cases}
\end{align}
for $\lambda_0>0$ large enough. We define the weight $m$ as
\begin{align}
\label{eq:m1} \begin{cases}
    \p_t m(t,k,\eta)= \sup\limits_{j\in \mathbb{Z}\setminus\{0\}} \left(\tfrac {10}{1+(\tfrac \eta j -t )^2 }\tfrac 1{\langle k-j \rangle^3 }\right) m(t,k,\eta), &\qquad  \text{if } (k,\eta)\in S_{t},\\
    \p_t  m(t,k,\eta)=0, &\qquad \text{otherwise},\\
    m(0,k,\eta)=1.
\end{cases}
\end{align}
This weight is used to absorb error terms that in the linear problem were integrable in time.

The standard weight is defined with $j=k$ and not only on $S_t$. However, for our problem the weight above is enough and it does simplify some proofs while keeping useful properties which are proved below.

To define the weights $J,\tilde J$ we introduce a minor modification of one of the main weights used in \cite{masmoudi2022stability} (the weight $\Theta$).
 Define the \textit{non-resonant} piece  as
\begin{align*}
    q_{NR}(t,\eta)&=\left(\tfrac { k^3}{ 2 \eta  }\left[1+ a_{k,\eta} \vert t-\tfrac \eta k \vert \right]\right)^{\rho +\tfrac 1 2   }q_{NR} (t_{k,\eta}^+,\eta) & \forall t&\in  \tilde I^R_{k,\eta },\\
    q_{NR}(t,\eta)&= \left(\left[1+ a_{k,\eta} \vert t-\tfrac \eta k \vert \right]\right)^{-\rho  }q_{NR}(\tfrac \eta k ,\eta) & \forall t&\in \tilde  I^L_{k,\eta }.
\end{align*}
The \textit{resonant} part is given by
\begin{align*}
q_R(t,\eta)&=\left(\tfrac { k^3}{ 2\eta  }\left[1+ a_{k,\eta} \vert t-\tfrac \eta k \vert \right]\right)^{-\frac 1 2 }q_{NR} (t,\eta )  &\forall t&\in   \tilde I_{k,\eta },
\end{align*}
with
\begin{align*}
    a_{k,\eta }&= 4(1-\tfrac 1 2 \tfrac {k^3} \eta ).
\end{align*}
On $ I^R_{k,\eta } \setminus \tilde I^R_{k,\eta } $ both weights are constant. We remark that these weights correspond to a modification of  $\Theta$ in \cite{masmoudi2022stability}, in particular  $\tfrac {q_{R}}{q_{NR}}= \tfrac {\sqrt{\Theta_{NR}}}{\sqrt{\Theta_{R}}}$.

Then, we define the weight
\begin{align}
\label{def:q}
    q(t, k,\eta ) =
    \begin{cases}
    q( t_\eta, \eta )& t<t_{\lfloor{ \eta }\rfloor^{\frac 1 3 } ,\eta},\\
    q_{NR}(t , \eta )&t\in [t_{\lfloor{ \eta }\rfloor^{\frac 1 3 },\eta} , 4\eta]\setminus   I_{k,\eta },  \\
    q_R(t, \eta )&t\in  I_{k,\eta }, \\
    1& t>2\eta. \\
    \end{cases}
\end{align}
Finally, we define $J,\tilde J$ as
\begin{align}
    \label{def:J}J(t,k,\eta)&= \frac {e^{8 \rho \vert \eta\vert^{\frac 1 3 }}} {q(t,k,\eta) } +  e^{8 \rho \vert k\vert^{\frac 1 3 }},\\
    \label{def:tJ}\tilde J(t,k,\eta)&= \frac {e^{8 \rho \vert \eta\vert^{\frac 1 3 }}} {q(t,k,\eta) }.
\end{align}
\subsubsection{Properties of the weights}
We now list some useful properties of the weights we defined above. For the weight $m$ we have the following.
\begin{lemma}
\label{lem:fullm1}
    Let $(k,\eta),(l ,\xi)\in \mathbb{Z}\times \mathbb{R}$. Then
    \begin{itemize}
    \item[i)] $1\le m\le \exp( \tfrac {\pi^3} {6} )$
    \item[ii)] If $k \neq 0$ and  $(l ,\xi)\in S_t$ then
       \begin{align*}
        \frac 1 {1+\vert t-\frac \eta k \vert}\lesssim \sqrt{\frac {\p_t m}{m}(t,l,\xi) }\langle  k-l,\eta-\xi \rangle ^3.
    \end{align*}
    \item[iii)] If $k \neq0$ and $\vert k,\eta\vert \le 2  \vert l,\xi\vert $, then
    \begin{align*}
        \frac {\vert k, \eta \vert } {k^2} \frac { 1 } {1+\vert t-\frac \eta k \vert^2 }\lesssim  \frac {\langle  \tfrac \eta {k^3}\rangle ^{\frac 1 2 } } {1+\vert t-\frac \eta k \vert }  \left(1+\sqrt { t}\sqrt{\frac {\p_t m }{m}  (t,l,\xi ) }\right)\langle k-l,\eta-\xi\rangle ^5  +\frac 1 {\langle t\rangle^2 }.\label{eq:mc1}
    \end{align*}
    \item[iv)] For all $t,\eta,\xi$ and $k \neq 0 $ it holds that
    \begin{align*}
    \vert m(t,k,\eta ) - m(t,k,\xi)\vert
    &\lesssim  \frac {\vert \eta-\xi\vert}{ |k|} .
\end{align*}
    \end{itemize}
\end{lemma}

The proof of this Lemma is postponed to the Appendix \ref{sec:appendix}.

For the weight $J$, since $q_R$ and $q_{NR}$ only differ by a power compared to the weight used in \cite{masmoudi2022stability}, it is straightforward to adapt the proofs of Lemmas 3.1, 3.4, 3.6 in \cite{masmoudi2022stability}  and obtain the following bounds on $q$.
\begin{lemma}
\label{lem:fullq}
Let $q$ be defined as in \eqref{def:q}. The following bounds holds true:
\begin{itemize}
    \item[i)] For $|\eta| >1$ and $k\in \mathbb{Z}$, there exists a $\mu>0$ such that
    \begin{align*}
    \frac {q(2\eta,k,\eta )}{q(0,k,\eta )}\approx \frac 1{\eta ^{\frac \mu{12} }}e^{\frac \mu 2 \vert \eta \vert^{\frac 1 3}}.
    \end{align*}
    \item[ii)]
        For $t\in \tilde  I_{k,\eta }$ and  we  obtain
        \begin{align*}
  \frac {\p_t q}q(t,\eta ) &\approx \frac 1 {1+\vert t-\frac \eta k \vert }  .
\end{align*}

    \item[iii)]  There exists a $\mu>0$ such that for all $t,\eta,\xi$
    \begin{align*}
    \frac {q_{NR}(t,\eta )}{q_{NR}(t,\xi )}\lesssim e^{\frac \mu 2 \vert \eta -\xi  \vert^{\frac 1 3}}.
    \end{align*}
\end{itemize}
\end{lemma}
\begin{remark}
    We remark that for property ii) in \cref{lem:fullq}, it is usually assumed that $t\geq \lfloor{|\eta|^s}\rfloor$, e.g. \cite{Bedrossian15}, where $s$ is the Gevrey regularity associated to the weight. This is because $\p_tq$ can go to zero as $k$ is close to  $|\eta|^s$ in view of the coefficients $a_{k,\eta}$ (and $b_{k,\eta}$ in \cite{Bedrossian15}) involved in the definition of the resonant and non-resonant part. In our case, we defined the weight in a slightly different way so to have the technical advantage that now $|a_{k,\eta}|\gtrsim 1$ for any $k,\eta$ of interest and therefore also $\p_t q$ enjoys the desired property.
\end{remark}

Thanks to the properties of $q$, we finally get the following bounds for $J$ and $\tilde J$.
\begin{lemma}
\label{lem:fullJ}
    Let $J$ and $ \tilde J$, be defined as in \eqref{def:J} and \eqref{def:tJ} respectively. Then, the following bounds holds true:
    \begin{itemize}
        \item[i)] For all $t,k,\eta,l ,\xi$ we have
            \begin{align*}
        \frac { J(t,k,\eta )}{J(t,l ,\xi )}\lesssim \left(1 + \textbf{1}_{I_{l ,\xi} } \frac {(\tfrac {|\xi|} {|l|^3})^{\frac 12 }  }{(1+\vert t-\frac \xi l\vert)^{\frac 1 2 } }\right ) e^{10\rho \vert k-l ,\eta-\xi\vert^{\frac 1 3 } }.
    \end{align*}
\item[ii)] If $t\in \tilde I_{k,\eta } \cap \tilde I_{l,\xi }^c $ then
    \begin{align*}
        \frac { J(t,k,\eta )}{J(t,l,\xi )}\lesssim \left(\tfrac  {|k|^3}{ |\eta|}\right)^{\frac 12 }    (1+\vert t-\tfrac \eta k \vert )^{\frac 12  }e^{10\rho  \vert k-l ,\eta-\xi\vert^{\frac 1 3 } }.
    \end{align*}
        \item[iii)]
    If $t\in \tilde I_{k,\eta }^c \cap \tilde I_{l,\xi }$ then
    \begin{align*}
        \frac { J(t,k,\eta )}{J(t,l ,\xi )}\lesssim    \frac{(\tfrac {|\xi|} {|l|^3})^{\frac 12 }}{(1+\vert t-\frac \xi l\vert)^{\frac 1 2 } }e^{10\rho \vert k-l ,\eta-\xi\vert^{\frac 1 3 } }
    \end{align*}
   \item[iv)] If $t\in \tilde I_{k,\eta } \cap \tilde I_{l,\xi } $ then
    \begin{align*}
        \frac { J(t,k,\eta )}{J(t,l ,\xi )}\lesssim \left(\frac  {|\xi k^3|} {|\eta l^3|}\right)^{\frac 12 }    \frac {(1+\vert t-\tfrac \eta k \vert )^{\frac 12  }} {(1+\vert t-\frac \xi l\vert)^{\frac 1 2 } }e^{10\rho  \vert k-l ,\eta-\xi\vert^{\frac 1 3 } }.
    \end{align*}
   \item[v)] If ( $t\in  \tilde I_{k,\eta }^c \cap \tilde I_{l,\xi }^c$ ) or ({ $k=l$} and $\eta \approx \xi $ ) then
    \begin{align*}
        \frac { J(t,k,\eta )}{J(t,l ,\xi )}\lesssim e^{10\rho  \vert k-l ,\eta-\xi\vert^{\frac 1 3 } }.
    \end{align*}
    \item[vi)] If $t \le \tfrac 1 2 \min \{|\eta|^{\frac 2 3 }, |\xi|^{\frac 2 3 }\}$, then there exists $C>0$ such that
    \begin{align*}
        \left \vert \frac {J(k,\eta )}{J(l,\xi )} -1 \right\vert \lesssim \frac {\langle k-l,\eta-\xi\rangle }{\vert k,l,\eta , \xi \vert }e^{C \rho \vert k-l,\eta-\xi\vert^{\frac 1 3 }}
    \end{align*}
    \item[vii)] If $4|\eta| \le \vert l\vert $, then
        \begin{align*}
    \left\vert \frac {J(k,\eta ) }{J(l,\xi ) }-1\right \vert &\lesssim  \frac {\langle k-l\rangle }{\vert k \vert^{\frac 23 } }e^{8\rho  \vert k-l\vert^{\frac 1 3 }}.
\end{align*}
\item[viii)]     Let $t\geq \min(|\xi|^\frac23,|\eta|^\frac23)$ and $|\eta|\geq C_1 |k|$ with $C_1\geq 4$, then we obtain the bound
    \begin{align*}
    \frac{|J(t,k,\eta)-J(t,k,\xi)|}{J(t,k,\xi)}\lesssim \vert\eta-\xi\vert \big( \frac{1}{|k|^{\frac 1 3}} + \sqrt{\tfrac {\p_t q }q (t,k,\eta )}\big)e^{C|\eta-\xi|^\frac13}.
\end{align*}
    \end{itemize}
    Furthermore, the estimates i)-v) hold with $J$ replaced by $\tilde J $

\end{lemma}
The proofs of most properties in this lemma are by now standard in the inviscid damping type problems as in \cite{Bedrossian15,masmoudi2022stability,bedrossian2023nonlinear}. They are all essentially based on \cite{Bedrossian15}. The only new estimate we need is the commutator-type estimate viii) above, which we prove in the appendix.

\subsection{Nonlinear estimates}
The goal of this section is to prove Proposition \eqref{prop:NLbounds}. We proceed by controlling each of the nonlinear terms defined in \eqref{def:NLphitoG}-\eqref{def:NLG} separately. We will use throughout this section the notation and the results of Section \ref{sec:weights}.
To ease the notation, we avoid writing the explicit $t$ dependence on  our unknowns and the weights.
\subsubsection{\textbf{Bound on the nonlinear action of $\phi$ on $G$}}
We aim at proving  \eqref{bd:NLphiG}.
First of all, we split the term $NL_{\phi \to G}$ defined in \eqref{def:NLphitoG} as follows
\begin{align*}
NL_{\phi\to G}
    &= \langle AG_{\neq} , \Delta^{-1}_t \p_y^tA( (\nabla^\perp \phi  \nabla) \p_y^t \phi )_{\neq} \rangle+ \langle AG_{\neq} , \Delta^{-1}_t \p_xA( (\nabla^\perp \phi  \nabla )\p_x \phi )_{\neq} \rangle\\
    &=:NL_{\phi\to G}^y+NL_{\phi\to G}^x.
\end{align*}
The most difficult term to bound is $NL_{\phi\to G}^y$ (due to the the time-dependent spatial derivatives) and therefore we estimate it in detail. The bounds for $NL_{\phi\to G}^x$ can be easily deduced from the ones we show for  $NL_{\phi\to G}^y$  and we thus omit the details for this term.

By Parseval's theorem, we infer
\begin{align*}
    \vert NL_{\phi\to G}^y\vert &\le\sum_{ \substack{k,l \\
    k\neq 0} }\iint  \left(
    \textbf{1}_{S_T}+\textbf{1}_{S_R}+\textbf{1}_{S_\calR}\right)\frac {\vert (\eta -kt )(\xi-lt)(\eta l-k\xi) \vert }{k^2 +(\eta -kt )^2 }  A(k,\eta)   \\
    &\hspace{2cm}\times \vert AG\vert (k,\eta) \vert \phi\vert (l,\xi) \vert \phi\vert (k-l ,\eta -\xi)  d\xi d \eta,\\
   & =: \widetilde{T}_{\phi\to G}+R_{\phi \to G}+\mathcal{R}_{\phi\to G},
\end{align*}
where we recall the definition of the transport, reaction and reminder sets in \eqref{def:SR},  \eqref{def:ST} and \eqref{def:ScalR} respectively.

We first rewrite the integral concentrated on the transport set as follows: since $\xi-lt=\eta-kt- (\eta -\xi-(k-l)t)$, by the triangle inequality we have
\[
\vert (\eta -kt )(\xi-lt) \vert  \leq  (\eta -kt )^2+\vert (\eta-\xi -(k-l )t )(\eta-kt)\vert.
\]
Then, with change of variables $\tilde{\xi}=\eta -\xi$ and $\tilde{l}=k-l$, notice that
\begin{align*}
    \eta -\xi-(k-l)t&= \tilde \xi -\tilde k t, \\
    \eta l-k\xi&=\eta (k-\tilde l ) - k(\eta-\tilde \xi) =-( \eta\tilde l-k\tilde \xi ),\\
    \textbf{1}_{S_T }(k,\eta, l,\xi) &= \textbf{1}_{S_R }(k,\eta,\tilde  l,\tilde \xi) .
\end{align*}
Thus, thanks to the observations above, we get that
\begin{align*}
    \widetilde{T}_{\phi \to G}&\leq \sum_{ \substack{k,l  \\ k \neq 0 }}\iint \textbf{1}_{S_T}  \frac {\vert (\eta -kt )^2(\eta l-k\xi) \vert }{k^2 +(\eta -kt )^2 }  A(k,\eta)  \vert AG\vert (k,\eta) \vert \phi\vert (l,\xi) \vert \phi\vert (k-l ,\eta -\xi) d\xi d\eta  \\
    &\hspace{-.5cm} +\sum_{ \substack{k,l  \\ k \neq 0 }}\iint\textbf{1}_{S_R}\frac { \vert (\tilde\xi-\tilde l t )(\eta -kt )( \eta\tilde l-k\tilde \xi )\vert  }{k^2 +(\eta -kt )^2 } A(k,\eta)  \vert AG\vert (k,\eta) \vert \phi\vert (k-\tilde l,\eta-\tilde \xi) \vert \phi\vert(\tilde l ,\tilde \xi)  d\tilde \xi d\eta \\
    &=: T_{\phi \to G}+R_{\phi \to G},
\end{align*}
where the term $R_{\phi \to G}$ is exactly as the one defined previously upon the identification of $\tilde\xi $ with $\xi$.
Consequently
\[
\vert NL_{\phi\to G}^y\vert\leq 2R_{\phi \to G}+  T_{\phi\to G}+\mathcal{R}_{\phi\to G}.
\]
In the sequel, we aim at showing the following bounds for each term on the right-hand side of the inequality above:
\begin{align}
    \label{bd:finRphitoG}
R_{\phi \to G}&\lesssim t \Vert A\nabla_t  G \Vert_{L^2}\Vert A \phi \Vert_{L^2}\left(\Vert A \phi \Vert_{L^2}+t^\frac12 \Vert \tilde A \sqrt{\tfrac {\p_t q }q}\phi  \Vert_{L^2}+t^\frac12\Vert  A \sqrt{\tfrac {\p_t m }{m}} \phi  \Vert_{L^2}\right),\\
\label{bd:finTphitoG}     T_{\phi\to G} &\lesssim  t \Vert A\nabla_t  G \Vert_{L^2} \Vert A\phi\Vert_{L^2}\left(t^\frac12 \Vert \tilde A\sqrt{\tfrac {\p_t q } q} \phi\Vert_{L^2}+\Vert A\phi\Vert_{L^2}\right),\\
\label{bd:fincalRphitoG} \mathcal{R}_{\phi \to G}&\lesssim \langle t \rangle^{-1} \Vert A \nabla_t   G \Vert_{L^2} \Vert A\phi\Vert_{L^2}^2.
\end{align}
With the bounds above at hand, we can directly prove a bound consistent with \eqref{bd:NLphiG}. Indeed
\begin{equation*}
    R_{\phi \to G}+T_{\phi \to G}+\calR_{\phi \to G}\lesssim t\eps\sqrt{\mathcal{D}_{G,\phi}}+ t^\frac32\eps\mathcal{D}_{G,\phi}.
\end{equation*}
Hence, integrating in time and using \eqref{bootstrap}, since $t\leq T\leq \delta \eps^{-3/2}$ we obtain
\begin{align*}
    \int_0^T \vert NL_{\phi\to G }^y\vert d\tau  &\lesssim T^\frac32 \eps^3\leq C\delta\eps^2,
\end{align*}
for a suitable constant $C>0$.

It thus remains to prove the bounds \eqref{bd:finRphitoG}-\eqref{bd:fincalRphitoG}, which is accomplished in the rest of this section.

\subsubsection*{$\bullet$ \textbf{Bound on $R_{\phi \to G}$}} Denoting the integrand of $R_{\phi \to G}$ as
\begin{align}
\label{def:calIRphiG}
    \calI_{R, \phi\to G} &:= \frac {\vert (\eta -kt )(\xi-lt)(\eta l-k\xi) \vert  }{k^2 +(\eta -kt )^2 } A(k,\eta) \vert AG\vert(k,\eta) \vert \phi\vert (l,\xi) \vert \phi\vert (k-l ,\eta -\xi)\
\end{align}
we split the reaction term according to
\begin{align*}
    &\sum_{\substack{k,l  \\ k \neq 0} }\iint \textbf{1}_{S_R} \left(\textbf{1}_{\tilde I_{k,\eta }\cap \tilde I_{k-l,\eta-\xi }^c}+\textbf{1}_{\tilde I_{k,\eta }^c\cap \tilde I_{k-l,\eta-\xi }^c} + \textbf{1}_{\tilde I_{k,\eta }\cap \tilde I_{k-l,\eta-\xi }}+ \textbf{1}_{\tilde I_{k,\eta }^c\cap \tilde I_{k-l,\eta-\xi }} \right)  \calI_{R, \phi\to G}\, d\xi d\eta \\
    &=:R_{\phi\to G }^{R,NR}+R_{\phi\to G }^{NR,NR}+R_{\phi\to G }^{R,R}+R_{\phi\to G }^{NR,R}.
\end{align*}
On $S_R$, where $(k-l,\eta-\xi)$ are the high frequencies, thanks to Lemma \ref{lem:useest} iii) we get
\begin{align}
    e^{\lambda(t) \vert k,\eta\vert^s}&\le e^{\lambda(t) \vert k-l,\eta-\xi\vert^s} e^{c_s\lambda(t) \vert l,\xi\vert^s}.\label{eq:Rexp}
\end{align}

\medskip
\noindent
\textit{$\diamond$ Bound on  $R_{\phi\to G }^{R,NR}$:} For this term, the support of the integral is $S_R\cap \tilde I_{k,\eta }\cap \tilde I_{k-l,\eta-\xi }^c$. By Lemma \ref{lem:fullm1} iii) with $l,\xi$ replaced by $k-l,\eta-\xi$,  we deduce that
\begin{align*}
    \frac {\vert (\xi-lt)(\eta l-k\xi) \vert }{k^2 +(\eta -kt )^2 }&\le t  \frac {\vert \eta, k \vert } {k^2} \tfrac 1 {1 +(\frac\eta k -t )^2}\vert l , \xi \vert^2 \\
    &\lesssim t  \frac { \langle \tfrac \eta{k^3}\rangle^{\frac 1 2 }} {1+\vert t-\frac \eta k \vert }  \left(1+\sqrt { t}\sqrt{\tfrac {\p_t m }{m}  (k-l,\eta-\xi )}\right)\vert l,\xi\vert^5+\frac 1{\langle t \rangle^2} \\
    &\lesssim \langle t\rangle   \frac { \langle \tfrac \eta{k^3}\rangle^{\frac 1 2 }} {1+\vert t-\frac \eta k \vert }  \left(1+\sqrt { t}\sqrt{\tfrac {\p_t m }{m}  (k-l,\eta-\xi )}\right)\vert l,\xi\vert^5,
\end{align*}
where on $\tilde I_{k,\eta }$ we absorbed the $\tfrac 1{\langle t \rangle^2}$ by $\frac { \langle \tfrac \eta{k^3}\rangle^{\frac 1 2 }}{1+\vert t-\frac \eta k \vert }\ge 1 $.
Then, by Lemma \ref{lem:fullJ} ii), we know that on $\tilde  I_{k,\eta}\cap \tilde I_{k-l ,\eta-\xi}^c $ we have
\begin{align*}
        \frac { J(k,\eta )}{J(k-l ,\eta - \xi )}\lesssim \frac  {|k|^{\frac 32} }{|\eta|^{\frac 1 2 }}   (1+\vert t-\tfrac \eta k \vert )^{\frac 1 2 }e^{20\rho  \vert l ,\xi\vert^{\frac 1 3 } }\vert l,\xi\vert^2 .
    \end{align*}
Therefore, by combining these estimates, using that on $ \tilde I_{k,\eta}$ we have $|\eta/k^3|\ge 1 $, and exploiting \eqref{eq:Rexp},  we infer
\begin{align*}
   & \frac {\vert (\eta -kt )(\xi-lt)(\eta l-k\xi) \vert }{k^2 +(\eta -kt )^2 } A(k,\eta)\\
    &\qquad \le \vert k\vert^{\frac 1 2 }  |\eta-kt|^\frac12\left(t+ {t^{\frac 32 } } \sqrt{\tfrac {\p_t m}{m} (k-l,\eta-\xi)}\right)  A(k-l,\eta-\xi) \vert l , \xi \vert^7 e^{\lambda(t)  \vert l ,\xi\vert^s }.
\end{align*}
Thus,  we deduce
\begin{align}
\label{bd:finRRNR}
     R_{\phi\to G }^{R,NR}  &\lesssim t \Vert A \nabla_t G \Vert_{L^2}\Vert A \phi \Vert_{L^2}\left( \Vert A\phi\Vert_{L^2}+ t^{\frac 1 2 }  \Vert \sqrt{\tfrac{\p_t m} {m }  }  A \phi \Vert_{L^2 }\right),
\end{align}
which is consistent with \eqref{bd:finRphitoG}.

\medskip
\noindent
\textit{$\diamond$ Bound on  $R_{\phi\to G }^{NR,NR} $: } We now have that the support of the integral is   $S_R\cap \tilde I_{k,\eta }^c\cap\tilde I_{k-l,\eta -\xi}^c$. Recalling the definition of \eqref{def:calIRphiG}, we first observe that
\begin{align*}
    \calI_{R,\phi\to G}=\frac {\vert (\eta -kt )(\xi-lt)(\eta l-k\xi) \vert  }{(k^2 +(\eta -kt )^2)^\frac32 } A(k,\eta) \vert \Lambda_tAG\vert(k,\eta) \vert \phi\vert (l,\xi) \vert \phi\vert (k-l ,\eta -\xi).
\end{align*}
Then
\begin{align*}
    \frac {|(\eta -kt )(\xi-lt)(\eta l-k\xi) |}{(k^2 +(\eta -kt )^2)^{\frac 32 } }\lesssim \frac {t\vert\eta, k\vert  }{k^2 +(\eta -kt )^2 } \vert l ,\xi\vert^2
\end{align*}
Combining Lemma \ref{lem:fullm1} iii) with the fact that
 on $\tilde I_{k-l,\eta-\xi}^c$ we have $\langle  t-\tfrac \eta k \rangle  \gtrsim \langle \tfrac \eta {k^3} \rangle  $, we deduce that
\begin{align*}
    \frac {\vert \eta,k \vert } {k^2} \frac { 1 } {1+ (\frac \eta k-t )^2 }\lesssim    \left(1+\sqrt { t}\sqrt{\tfrac {\p_t m }{m}  (k-l,\eta-\xi ) }\right)\vert l,\xi\vert^3 .
\end{align*}
Moreover, using Lemma \ref{lem:fullJ} v), we infer
\begin{align*}
    J(k,\eta ) &\lesssim J(l,\xi  )J(k-l,\eta-\xi  ) e^{20 \rho \vert l,\xi\vert^{s}}.
\end{align*}
Therefore, combining these estimates and using \eqref{eq:Rexp} we get
\begin{align*}
    \calI_{R, \phi\to G} &\lesssim t  \left(1+\sqrt { t}\sqrt{\tfrac {\p_t m }m  (k-l,\eta-\xi ) }\right)\langle l,\xi\rangle^{-2}  \vert AG\vert(k,\eta) \vert  A \phi\vert (l,\xi) \vert A \phi\vert (k-l ,\eta -\xi),
\end{align*}
from which we deduce
\begin{align}
\label{bd:finNRNRR}
    \vert R_{\phi\to G }^{NR,NR} \vert &\lesssim t \Vert A \nabla_t G \Vert_{L^2}\Vert A \phi \Vert_{L^2}\left( \Vert A\phi\Vert_{L^2}+ t^{\frac 1 2 }  \Vert \sqrt{\tfrac{\p_t m} {m }  }  A \phi \Vert_{L^2 }\right),
\end{align}
which is consistent with \eqref{bd:finRphitoG}.

\medskip
\noindent
\textit{$\diamond$ Bound on  $R_{\phi\to G }^{R,R}$:}
For this term, recalling the definition of $\calI_{R,\phi \to G}$ we first observe that
\[
|\eta-kt||AG|(k,\eta)\leq |A \nabla_tG|(k,\eta).
\]
Then, in the support of the integral, that is  $S_R\cap\tilde I_{k,\eta}\cap\tilde I_{k-l,\eta-\xi}$, by Lemma \ref{lem:fullJ} iv) we infer
\begin{align*}
        \frac { J(k,\eta )}{J(k-l ,\eta - \xi )}\lesssim  \left(\frac {1+\vert t-\frac \eta k \vert }{1+\vert t-\frac {\eta-\xi} {k-l } \vert}\right )^{\frac 1 2} e^{10\rho  \vert l ,\xi\vert^{\frac 1 3 } }\vert l,\xi\vert^3.
\end{align*}
To control the remaining multipliers, since $t\in \tilde I_{k,\eta}\cap \tilde I_{k-l,\eta-\xi}$, it is convenient to write  $$t=\frac \eta k +\tau_1=\frac {\eta-\xi }{k-l} +\tau_2$$ and observe that
\begin{align*}
    \eta l-k\xi&= \eta (k-l) - (\eta-\xi) k = \left( {\frac \eta k -t +t -\frac {\eta-\xi}{k-l}}\right){k(k-l)}\\
    &=  {(\tau_2-\tau_1)}{k(k-l)}.
\end{align*}
Thus
\begin{align*}
    &\frac {\vert (\xi-lt)(\eta l-k\xi) \vert }{k^2 +(\eta -kt )^2 }\frac { J(t,k,\eta )}{J(t,k-l ,\eta - \xi )}\lesssim t \frac {| \tau_1|+|\tau_2| } {(1+\vert t-\frac \eta k \vert )^{\frac 3 2 }(1+\vert t-\frac {\eta-\xi} {k-l } \vert)^{\frac 1 2 }}e^{10\rho  \vert l ,\xi\vert^{\frac 1 3 } }\vert l,\xi\vert^5\\
    &\quad \lesssim t \left(\frac {1} {(1+\vert t-\frac \eta k \vert )^{\frac 1 2 }(1+\vert t-\frac {\eta-\xi} {k-l } \vert)^{\frac 1 2 }}+ \frac {\vert \tau_2\vert^{\frac 1 2 }  } {(1+\vert t-\frac \eta k \vert )^{\frac 3 2 }}\right)e^{10\rho  \vert l ,\xi\vert^{\frac 1 3 } }\vert l,\xi\vert^5.
\end{align*}
Since $|\tau_2|\leq (|(\eta-\xi)/(k-l )^3|)^\frac12$ for $t\in \tilde I_{k-l ,\eta-\xi}$, we estimate
\begin{align*}
    \frac {\vert \tau_2\vert^{\frac 1 2 }  } {1+\vert t-\frac \eta k \vert }&\le \frac {(\tfrac {|\eta| }{|k|^3} )^{\frac 1 2 } } {1+\vert t-\frac \eta k \vert }(\textbf{1}_{\vert k-l, \eta-\xi\vert \le \langle t\rangle ^2 }+\textbf{1}_{\vert k-l, \eta-\xi\vert \ge \langle t\rangle ^2 })\vert l ,\xi \vert^2 \\
    &\le \frac {\langle \tfrac \eta k \rangle^{\frac 1 2 }   } {1+\vert t-\frac \eta k \vert }\textbf{1}_{\vert k-l, \eta-\xi\vert \le \langle t\rangle ^2 } \vert l ,\xi \vert^2  \\
    &\quad +\frac 1 {\vert \eta-\xi , k-l \vert^{\frac 1 2 } }\frac {\vert \eta, k \vert  }{\vert k \vert } \frac {1  } {1+\vert t-\frac \eta k \vert } \textbf{1}_{\vert k-l, \eta-\xi\vert \ge \langle t\rangle ^2 }\vert l ,\xi \vert^3\\
    &\le \left(1+ \sqrt t \sqrt{\tfrac {\p_t m } {m} (k-l,\eta-\xi)}\right)\vert l ,\xi \vert^6,
\end{align*}
where in the last inequality we used  Lemma  \ref{lem:fullm1} ii).
Overall, we have the bound \begin{align*}
    \frac {\vert (\xi-lt)(\eta l-k\xi) \vert }{k^2 +(\eta -kt )^2 }\frac { J(t,k,\eta )}{J(t,k-l ,\eta - \xi )}&\lesssim \left(1+ \sqrt t \sqrt{\tfrac {\p_t m }{m} (k-l,\eta-\xi)}\right)e^{10\rho  \vert l ,\xi\vert^{\frac 1 3 } }\vert l,\xi\vert^{11},
\end{align*}
meaning that also for this term, using \eqref{eq:Rexp}, we get
\begin{align}
\label{bd:finRRR}
     R_{\phi\to G }^{R,R}  &\lesssim t \Vert A \nabla_t G \Vert_{L^2}\Vert A \phi \Vert_{L^2}\left( \Vert A\phi\Vert_{L^2}+ t^{\frac 1 2 }  \Vert \sqrt{\tfrac{\p_t m} {m }  }  A \phi \Vert_{L^2 }\right),
\end{align}
which is consistent with \eqref{bd:finRphitoG}.

\medskip
\noindent
\textit{$\diamond$ Bound on $R_{\phi\to G }^{NR,R}$:}
It remains to control this last term, where the support of the integral is   $S_R\cap\tilde I_{k,\eta }^c\cap \tilde I_{k-l,\eta -\xi}$. According to definition of $J, \tilde J$ \eqref{def:J}-\eqref{def:tJ}, we split $J(k,\eta ) = \tilde J(k,\eta ) + e^{8 \rho \vert k\vert^{\frac 1 2 }}$. For the second term we use
\begin{align*}
 e^{8 \rho \vert k\vert^{\frac 1 2 }}&\le J(k-l, \eta-\xi) e^{8 \rho \vert l\vert^{s}}    \
\end{align*}
and Lemma \ref{lem:fullm1} iii) to get
\begin{align}
    \notag&\frac {\vert (\eta -kt )(\xi-lt)(\eta l-k\xi) \vert }{k^2 +(\eta -kt )^2 }e^{8 \rho \vert k\vert^{\frac 1 2 }}\lesssim t|\eta-kt|\frac {\vert \eta ,k \vert}{k^2+(\eta-kt)^2}   J(k-l, \eta-\xi) e^{20 \rho \vert l,\xi\vert^{s}} \vert l,\xi\vert^2\\
    \notag&\qquad \lesssim t|\eta-kt| \left( \frac {\langle \frac  \eta {k^3}  \rangle }{1+\vert\frac\eta k-t\vert}\left(1+ \sqrt t \sqrt{\tfrac {\p_t m }{ m} (k-l,\eta-\xi)}\right)+\frac 1 {\langle t \rangle^2}\right)  J(k-l, \eta-\xi) e^{20 \rho \vert l,\xi\vert^{s}} \vert l,\xi\vert^2\\
   \label{eq:NRR2}&\qquad \lesssim t|\eta-kt| \left(1+ \sqrt t \sqrt{\tfrac {\p_t m }{ m} (k-l,\eta-\xi)}\right)  J(k-l, \eta-\xi) e^{20 \rho \vert l,\xi\vert^{s}} \vert l,\xi\vert^2.
\end{align}
For $\tilde J$, appealing to Lemma \ref{lem:fullJ} iii)  we obtain
\begin{align*}
    \tilde J(k,\eta ) &\lesssim\frac{ \vert \tfrac {\eta -\xi }{(k-l)^3}\vert^{\frac 1 2}}{(1+ \vert t-\frac {\eta-\xi}{k-l}\vert)^{\frac 1 2 } }\tilde J(k-l,\eta-\xi  ) e^{20 \rho \vert l,\xi\vert^{s}}.
\end{align*}
Since on $I_{k,\eta }^c$ it holds that $\vert\eta, k\vert/ |k|^3 \lesssim  \langle  \frac\eta k -t \rangle $, by Lemma \ref{lem:fullq} ii) we have
\begin{align*}
    &\frac {\vert (\eta -kt )(\xi-lt)(\eta l-k\xi) \vert }{k^2 +(\eta -kt )^2 }\frac{\vert \tfrac {\eta -\xi }{(k-l)^3}\vert^{\frac 1 2}} {(1+ \vert t-\frac {\eta-\xi}{k-l}\vert)^{\frac 1 2 } }\\
    &\quad \lesssim t |\eta-kt| \frac {\vert\eta, k\vert  }{k^2 }\frac 1 {1 +(\frac\eta k -t )^2 } \left( 1+ \sqrt t \frac 1 {(1+ \vert t-\frac {\eta-\xi}{k-l}\vert)^{\frac 1 2 } }\right) \vert l ,\xi\vert^2 \\
    &\quad \lesssim  t|\eta-kt|  \left( 1+ \sqrt t \frac 1 {(1+ \vert t-\frac {\eta-\xi}{k-l}\vert)^{\frac 1 2 } }\right) \vert l ,\xi\vert^3 \\
    &\quad \lesssim t |\eta-kt| \left( 1+ \sqrt t\sqrt{\tfrac {\p_t q }q(k-l,\eta-\xi) }\right) \vert l ,\xi\vert^6.
\end{align*}
Therefore, we infer
\begin{align}\begin{split}
    &\frac {\vert (\eta -kt )(\xi-lt)(\eta l-k\xi) \vert }{k^2 +(\eta -kt )^2}\tilde J(k,\eta) \\
    &\quad \lesssim  t |\eta-kt|  \left( 1+ \sqrt t\sqrt{\tfrac {\p_t q }q(k-l,\eta-\xi) }\right) \tilde J(k-l,\eta-\xi  ) e^{20 \rho \vert l,\xi\vert^{s}}\vert l,\xi\vert^9. \label{eq:NRR1}
\end{split}\end{align}
Combining the bounds \eqref{eq:Rexp}, \eqref{eq:NRR2} and \eqref{eq:NRR1} we deduce  that
\begin{align}
\label{bd:finRNRR}
     R_{\phi\to G }^{NR,R}&\lesssim t \Vert A \nabla_t G \Vert_{L^2}\Vert A \phi \Vert_{L^2}\left(\Vert A \phi \Vert_{L^2}+t^\frac12 \Vert \tilde A \sqrt{\tfrac {\p_t q }q}\phi  \Vert_{L^2}+t^\frac12\Vert  A \sqrt{\tfrac {\p_t m }{m}} \phi  \Vert_{L^2}\right)
\end{align}
which is consistent with \eqref{bd:finRphitoG}.

\subsubsection*{$\bullet$ \textbf{Bound on  $T_{\phi\to G}$:}}
Recall that we need to estimate
\begin{align*}
    T_{\phi\to G}=\sum_{ \substack{k,l  \\ k \neq 0 }}\iint \textbf{1}_{S_T}  \frac {\vert (\eta -kt )^2(\eta l-k\xi) \vert }{k^2 +(\eta -kt )^2 }  A(k,\eta)  \vert AG\vert (k,\eta) \vert \phi\vert (l,\xi) \vert \phi\vert (k-l ,\eta -\xi) d\xi d\eta.
\end{align*}
First of all, we bound $(\eta-kt)^2/(k^2+(\eta-kt)^2)\leq 1$. Furthermore, we write
\begin{align*}
    \eta l-k \xi &= (\eta -\xi-(k-l)t)k -  (\eta-kt)(k-l),
\end{align*}
which yields the  bound
\begin{align*}
    \vert \eta l-k\xi\vert   &\le t\vert k,  \eta-kt \vert \vert k-l, \eta -\xi  \vert.
\end{align*}
Using Lemma \ref{lem:fullJ} i), we obtain that
\begin{align*}
   {J(k,\eta ) }&\lesssim \left( J(l,\xi)+ \textbf{1}_{I_{l,\xi} } \sqrt{\tfrac \xi{l^3}}\tfrac 1 {1+\vert t-\frac \xi l\vert^{\frac 1 2 } }  \tilde J(l,\xi) \right)e^{16 \rho \vert k-l ,\eta-\xi\vert^s  }\\
   &\lesssim \left( J(l,\xi)+  \sqrt{t}\sqrt{\tfrac {\p_t q} q (l ,\xi)}  \tilde J(l,\xi) \right)e^{16 \rho \vert k-l ,\eta-\xi\vert^s  }.
\end{align*}
Furthermore, as in the estimate \eqref{eq:Rexp} for $S_T$, we note that
\begin{align}
    e^{\lambda(t) \vert k,\eta\vert^s}&\le e^{c_s\lambda(t) \vert k-l,\eta-\xi\vert^s} e^{\lambda(t) \vert l,\xi\vert^s}.\label{eq:Texp}
\end{align}
We thus infer
\begin{align*}
    \frac {\vert (\eta -kt )^2(\eta l-k\xi) \vert }{k^2 +(\eta -kt )^2 }  A(k,\eta)&\lesssim t\vert k,  \eta-kt \vert \langle k-l,\eta-\xi\rangle^{-2}  A(k-l,\eta-\xi)\\
    &\qquad \qquad \times  \left( A(l,\xi)+  \sqrt{t}\sqrt{\tfrac {\p_t q} q (l ,\xi)}  \tilde A(l,\xi) \right).
\end{align*}
This yields the bound
\begin{align*}
    T_{\phi\to G}&\lesssim t \Vert A \nabla_t G \Vert_{L^2} \Vert A\phi\Vert_{L^2}\left(\sqrt t \Vert \tilde A\sqrt{\tfrac {\p_t q } q} \phi\Vert_{L^2}+\Vert A\phi\Vert_{L^2}\right),
\end{align*}
 which proves the desired estimate \eqref{bd:finTphitoG}.

\subsubsection*{$\bullet$ \textbf{Bound on $\calR_{\phi\to G}$:}} In this section we bound the remainder of the action from $\phi$ to $G$
\begin{align*}
    \calR_{\phi\to G} &=\sum_{ \substack{k,l \\
    k\neq 0} }\iint \textbf{1}_{S_\calR}\frac {\vert (\eta -kt )(\xi-lt)(\eta l-k\xi) \vert }{k^2 +(\eta -kt )^2 }  A(k,\eta)    \vert AG\vert (k,\eta) \vert \phi\vert (l,\xi) \vert \phi\vert (k-l ,\eta -\xi)  d\xi d \eta\\
    &\le \sum_{ \substack{k,l \\
    k\neq 0} }\iint \textbf{1}_{S_\calR}\frac {\vert (\xi-lt)(\eta l-k\xi) \vert }{k^2 +(\eta -kt )^2 }  A(k,\eta)    \vert A \Lambda_t G\vert (k,\eta) \vert \phi\vert (l,\xi) \vert \phi\vert (k-l ,\eta -\xi)  d\xi d \eta.
\end{align*}
On $S_\calR$ we obtain the estimate
\begin{align*}
    A(k,\eta ) \lesssim \, & \frac 1{\langle k-l,\eta-\xi \rangle^{\frac N2}\langle l,\xi \rangle^{\frac N2} }\frac {J(k,\eta ) }{J(l,\xi )J(k-l,\eta-\xi)} \\
   & \times e^{\lambda(t)(\vert k,\eta\vert^s- \vert l,\xi\vert^s-\vert k- l, \eta-\xi\vert^s)}A(l,\xi)A(k-l,\eta-\xi).
\end{align*}
Moreover, by Lemma \ref{lem:fullJ} i) we  deduce
\begin{align*}
    \frac {J(k,\eta ) }{J(l,\xi )J(k-l,\eta-\xi)}\lesssim \langle  l,\xi\rangle  \langle k- l, \eta- \xi\rangle  e^{8\rho(\vert l,\xi\vert^s+ \vert k-l,  \eta-\xi\vert^s)}).
\end{align*}
Furthermore, since $s\le \tfrac 1 2 $ by Lemma \ref{lem:useest} iii), we have that $$\vert k,\eta\vert^s\le \tfrac 5 6( \vert l,\xi\vert^s+\vert k- l, \eta-\xi\vert^s),$$ and therefore
\begin{align*}
    \frac {J(k,\eta ) }{J(l,\xi )J(k-l,\eta-\xi)} e^{\lambda(t)(\vert k,\eta\vert^s- \vert l,\xi\vert^s-\vert k- l, \eta-\xi\vert^s)}\le \langle  l,\xi\rangle  \langle k- l, \eta- \xi\rangle.
\end{align*}
Hence
\begin{align*}
    \frac {\vert (\xi-lt)(\eta l-k\xi) \vert }{k^2 +(\eta -kt )^2 }A(k,\eta ) &\lesssim \langle t \rangle^{-1} \langle k-l,\eta-\xi \rangle^{-2}A(l,\xi)A(k-l,\eta-\xi)
\end{align*}
and so we infer
\begin{align*}
    \calR_{\phi\to G} \lesssim \langle t \rangle^{-1} \Vert A \nabla_t G \Vert_{L^2} \Vert A\phi\Vert_{L^2}^2,
\end{align*}
which gives \eqref{bd:fincalRphitoG}.

\subsubsection{\textbf{Bound on the nonlinear action of $G$ on $\phi$}}
\label{sec:phi}
In this section, we provide bounds for the term defined in \eqref{def:NLGtophi}. We rewrite and split this term as follows
\begin{align*}
     \vert NL_{G\to \phi}\vert &=\left| \langle  A( (\nabla^\perp G_{\neq}\cdot \nabla) \phi)- (\nabla^\perp G_{\neq}\cdot \nabla) A\phi,A \phi \rangle\right|\\
    &\leq
    \sum_{ \substack{k,l \\
    k\neq l} }\iint  \left(
    \textbf{1}_{S_R}+\textbf{1}_{S_T}+\textbf{1}_{S_\calR}\right)\vert \eta l -k \xi\vert  \vert A(k,\eta)-A(l,\xi)\vert\\
&\hspace{2cm}\times\vert A\phi\vert (k,\eta ) \vert G\vert (k-l,\eta-\xi) \vert \phi\vert (l,\xi) d\xi d\eta\\
     &= R_{G\to\phi}+T_{G\to\phi}+\calR_{G\to\phi},
\end{align*}
where we recall the notation introduced in \eqref{def:SR}-\eqref{def:ScalR}.

Our next goal is to prove the following bounds
\begin{align}
    \label{bd:finRGtophi}
R_{G \to \phi}&\lesssim \langle t\rangle  \Vert A \nabla_t G \Vert_{L^2}\Vert A\phi \Vert_{L^2}\left(\langle t\rangle^{\frac 1 2}\Vert A\sqrt{\tfrac{\p_t m}{m }} \phi \Vert_{L^2}+\Vert A\phi \Vert_{L^2}\right)\\
\label{bd:finTGtophi}    T_{G \to \phi} &\lesssim \Vert A \Lambda^{\frac s 2 } \phi\Vert_{L^2 }^2  \Vert A G\Vert_{L^2 } +\langle t \rangle ^{\frac 54-\frac {3s}4 } \Vert  A\Lambda^{\frac s 2 }\phi\Vert_{L^2}\Vert A \nabla_t G \Vert_{L^2}\Vert  A\phi\Vert_{L^2}\\
\label{bd:fincalRGtophi} \mathcal{R}_{G \to \phi}&\lesssim \langle t \rangle^{-1} \Vert A \nabla_t G \Vert_{L^2} \Vert A\phi\Vert_{L^2}^2.
\end{align}
As in the previous section, it is straightforward to check that, thanks to \eqref{bootstrap} and $\eps t^\frac32\leq \delta$, one can prove \eqref{bd:NLGphi}. Indeed, for $\mathcal{R}_{\phi \to G}$ it is obvious. For the reaction term, appealing to \eqref{bootstrap} we have
\begin{align*}
    R_{G \to \phi}&\lesssim \eps \langle t \rangle ^{3/2} \Vert A \nabla_t G \Vert_{L^2}\Vert A\sqrt{\tfrac{\p_t m}{m }} \phi \Vert_{L^2}+\langle t \rangle \Vert A \nabla_t G \Vert_{L^2}\Vert A\phi \Vert_{L^2}^{2}\\
    &\lesssim t^\frac32 \eps  \mathcal{D}_{G,\phi}+t\eps^2  \sqrt{\mathcal{D}_{G,\phi}}
\end{align*}
and thus integrating in time we get an estimate consistent with \eqref{bd:NLGphi} for this term.
For the transport term it is enough to observe that $\eps t^{\frac54-\frac{3s}{4}}\leq \delta t^{-\frac{1+3s}{4}}$. Since $s>1/3$ we also get that $(1+3s)/4\geq (1+\gamma_*)/2$ by the choice of $\gamma_*$ in \eqref{def:lambda}. This implies
\begin{align*}
    T_{G\to \phi}\lesssim \delta \mathcal{D}_{G,\phi},
\end{align*}
and thus integrating in time we get \eqref{bd:NLGphi}.

\subsubsection*{$\bullet$ \textbf{Bound on  $R_{ G\to \phi}$:}}
Let us rewrite the integrand of this term as
\begin{align*}
    \calI_{R,G\to \phi}&= \frac {\vert \eta l -k \xi\vert\vert A(k,\eta)-A(l,\xi)\vert}{\vert k-l\vert+ \vert \eta -\xi-(k-l)t \vert }     \vert A\phi\vert (k,\eta ) \vert \Lambda_t  G\vert (k-l,\eta-\xi) \vert \phi\vert (l,\xi).
\end{align*}
We then split the reaction term by
\begin{align*}
    R_{G\to \phi}&=\sum_{ \substack{k,l \\
    k\neq l} }\iint \textbf{1}_{S_R }( \textbf{1}_{\tilde I_{k-l,\eta-\xi}^c}+\textbf{1}_{\tilde I_{k-l,\eta-\xi}}) I_{R,G\to \phi} \, d\xi d\eta\\
    &=: R^{NR}_{G\to\phi}+R^R_{G\to\phi}.
\end{align*}
We recall that on $S_R$ we have the bound \eqref{eq:Rexp}, which give us space to pay regularity on the low frequencies $(l,\xi).$

\medskip
\noindent
\textit{$\diamond$ Bound on  $R_{G\to\phi}^{NR}$:} In the set $S_R\cap I_{k-l,\eta-\xi}^c$, by Lemma \ref{lem:fullJ} we obtain
\begin{align*}
        \frac { J(k,\eta )}{J(k-l ,\eta - \xi )}\lesssim e^{10\rho  \vert l ,\xi\vert^{\frac 1 3 } } \vert l,\xi\vert^2 .
    \end{align*}
Therefore, combining the bound above with    \begin{align*}
        \frac {\vert \eta l -k \xi\vert}{\vert k-l\vert+ \vert \eta -\xi-(k-l)t \vert }&\lesssim \left(\frac t{1+ \vert\frac { \eta -\xi}{k-l}-t \vert }+1 \right)\vert l ,\xi \vert ^2\\
        &\lesssim \langle t\rangle \vert l ,\xi \vert ^2 ,
    \end{align*}
    and using \eqref{eq:Rexp}, we infer
\begin{align}
    \vert R_{G\to\phi}^{NR}\vert &\lesssim \langle t \rangle \Vert A \nabla_t G \Vert_{L^2}\Vert A \phi  \Vert^2_{L^2},\label{eq:GpR1}
\end{align}
which is consistent with \eqref{bd:finRGtophi}.

\medskip
\noindent
\textit{$\diamond$ Bound on $R_{G\to\phi}^R$:} We consider the set $S_R $ with $t\in \tilde  I_{k-l,\eta-\xi} $ and therefore by Lemma \ref{lem:fullJ} we get
\begin{align*}
        J(t,k,\eta )\lesssim    \frac {\vert {\tfrac   {\eta-\xi}{(k-l)^3 }}\vert ^{\frac 1 2 } }{(1+\vert t-\frac {\eta-\xi} {k-l} \vert)^{\frac 1 2 } } J(t,k-l ,\eta-\xi ) e^{10\rho  \vert l ,\xi\vert^{\frac 1 2 } }.
    \end{align*}
Moreover, in $ \tilde  I_{k-l,\eta-\xi} $ we also know that $\vert \tfrac {\eta-\xi} {k-l }\vert \le 2 t $. Therefore, when $\vert \eta, k \vert \le t^2$, we estimate
\begin{align*}
    \textbf{1}_{\vert \eta, k \vert \le t^2}\frac {\vert \eta l -k \xi\vert \left\vert {\tfrac   {\eta-\xi}{(k-l )^3}}\right\vert^{\frac 1 2 }  }{\vert k-l\vert   (1+\vert\frac {\eta-\xi} {k-l } \vert )^{\frac 3 2 }}
    &\lesssim  \textbf{1}_{\vert \eta, k \vert \le t^2}   \frac {\frac {\vert \eta-\xi\vert^{\frac 3 2 }} {\vert k-l\vert ^{\frac 52 } }} {(1+\vert\frac {\eta-\xi} {k-l } \vert )^{\frac 3 2 }} {\vert l,\xi\vert^2 }\\
    &\lesssim \textbf{1}_{\vert \eta, k \vert \le t^2}  t^{\frac 3 2 } \sqrt{\tfrac {\p_t m }m(k,\eta)}\vert l,\xi\vert^4 .
\end{align*}
When $\vert \eta, k \vert \ge t^2$ we instead have
\begin{align*}
     \frac {t^2}{\vert k-l\vert} & \le  \frac {\vert \eta ,k \vert }{\vert k-l\vert}\le \left( \big\vert \frac {\eta-\xi} {k-l }\big\vert+1\right) \vert l,\xi\vert \lesssim \langle t\rangle  \vert l,\xi\vert.
\end{align*}
Thus
\begin{align*}
    \textbf{1}_{\vert \eta, k \vert \ge t^2}\frac {\vert \eta l -k \xi\vert \left\vert {\tfrac   {\eta-\xi}{(k-l )^3}}\right\vert^{\frac 1 2 }  }{\vert k-l\vert   (1+\vert\frac {\eta-\xi} {k-l } \vert )^{\frac 3 2 }}
    &\lesssim  \textbf{1}_{\vert \eta, k \vert \ge t^2}  \frac {\tfrac {\vert \eta-\xi\vert^{\frac 3 2 }} {\vert k-l\vert ^{\frac 52 } } } {(1+\vert\frac {\eta-\xi} {k-l } \vert )^{\frac 3 2 }} {\vert l,\xi\vert^2 }
    \lesssim  \textbf{1}_{\vert \eta, k \vert \ge t^2} \frac {t^{\frac 3 2 }} {\vert k-l\vert  }   {\vert l,\xi\vert^2 }\\
    &\lesssim \textbf{1}_{\vert \eta, k \vert \ge t^2} \langle t \rangle^{\frac{1}{2}}   \vert l,\xi\vert^3.
\end{align*}
Therefore, by combining these estimates and using \eqref{eq:Rexp}, we obtain
\begin{align*}
    \frac {\vert \eta l -k \xi\vert\vert A(k,\eta)-A(l,\xi)\vert}{\vert k-l\vert+ \vert \eta -\xi-(k-l)t \vert } &\lesssim \left(\langle t\rangle ^{\frac 1 2 }
 + t^{\frac 3 2 } \sqrt{\tfrac {\p_t m }m(k,\eta)}\right) \vert l,\xi\vert^4 A(k-l ,\eta- \xi) e^{10\rho  \vert l ,\xi\vert^{\frac 1 2 } }.
\end{align*}
From the bound above, we deduce
\begin{align}\label{eq:GpR2}
    \vert R_{G\to\phi}^R\vert &\lesssim \langle t\rangle  \Vert A \nabla_t G \Vert_{L^2}\Vert A\phi \Vert_{L^2}(t^\frac12\Vert A\sqrt{\tfrac{\p_t m}{m }}\phi \Vert_{L^2}+ \Vert A\phi \Vert_{L^2}^2),
\end{align}
which is consistent with \eqref{bd:finRGtophi}.

\subsubsection*{$\bullet$ \textbf{Bound on $T_{G\to\phi}$:} } Here we want to estimate
\begin{align*}
    T_{G\to\phi} =\sum_{ \substack{k,l \\
    k\neq l} }\iint  \textbf{1}_{S_T} \vert \eta l -k \xi\vert \vert A(k,\eta)-A(l,\xi)\vert  \vert A\phi \vert (k,\eta ) \vert G\vert (k-l,\eta-\xi) \vert \phi\vert (l,\xi) d\xi d\eta.
\end{align*}
Recalling the definition of the weight in \eqref{def:A0}, we split the difference as
\begin{align*}
    A(k,\eta)-A(l,\xi)&= A(l,\xi )(e^{\lambda(t) (\vert k,\eta \vert^s-\vert l,\xi  \vert^s)}-1  )\\
    &+A(l,\xi )e^{\lambda(t) (\vert k,\eta \vert^s-\vert l,\xi  \vert^s)}\big(\frac{\langle k,\eta\rangle^N }{\langle l,\xi\rangle^N}-1  \big)\\
    &+A(l,\xi )e^{\lambda(t) (\vert k,\eta \vert^s-\vert l,\xi  \vert^s)}\frac{\langle k,\eta\rangle^N }{\langle l,\xi\rangle^N}\big(\frac{m(k,\eta )  }{m(l,\xi) }-1  \big)\\
    &+A(l,\xi )e^{\lambda(t) (\vert k,\eta \vert^s-\vert l,\xi  \vert^s)}\frac{\langle k,\eta\rangle^N }{\langle l,\xi\rangle^N}\frac {m(k,\eta )}{m(l,\xi)}\big(\frac{J(k,\eta)}{J(l,\xi)}-1  \big).
\end{align*}
Furthermore, on $S_T$ by Lemma \ref{lem:useest} iii) we know that there exists a constant $c_s\in (0,1) $ such that
\begin{align*}
    e^{\lambda(t) (\vert k,\eta \vert^s-\vert l,\xi  \vert^s)}\lesssim e^{c_s \lambda(t) \vert k- l,\eta- \xi  \vert^s}.
\end{align*}
Denoting the integrand
\begin{align*}
    \calI_{T,G\to \phi}&=\vert \eta l -k \xi\vert e^{c_s \lambda(t) \vert k- l,\eta- \xi  \vert^s}  \vert A\phi \vert (k,\eta ) \vert G\vert (k-l,\eta-\xi) \vert  A\phi\vert (l,\xi),
\end{align*}
we then split the transport term as
\begin{align*}
    T_{G\to\phi} &\lesssim
   \sum_{ \substack{k,l \\
    k\neq l} } \iint \textbf{1}_{\Omega_T} \calI_{T,G\to \phi}\bigg(\big\vert \frac{J(k,\eta)}{J(l,\xi)}-1 \big\vert  +\big\vert \frac{m(k,\eta )  }{m(l,\xi) }-1   \big\vert\\
        &\qquad  \qquad   +  e^{-c_s\lambda(t) \vert k- l,\eta- \xi  \vert^s}\vert e^{\lambda(t) (\vert k,\eta \vert^s-\vert l,\xi  \vert^s)}-1\vert +\big\vert \frac{\langle k,\eta\rangle^N }{\langle l,\xi\rangle^N}-1 \big\vert\bigg) d\xi d \eta \\
    &=: T_{G\to\phi,1}+ T_{G\to\phi,2}+ T_{G\to\phi,3}+ T_{G\to\phi,4}.
\end{align*}
We proceed by controlling each of the terms above separately.

\medskip
\noindent
\textit{$\diamond$ Bound on $T_{G\to \phi,1}$:} For $(t\le \tfrac 1 2 \min(|\xi|^{\frac 2 3 }, |\eta|^{\frac 2 3 }) )$ or ($t\ge \tfrac 1 2 \min(|\xi|^{\frac 2 3 }, |\eta|^{\frac 2 3 }) $ and $|\eta|\le 4|k| $), using  Lemma \ref{lem:fullJ} vi) or vii) respectively we get
\begin{align*}
   \big\vert \frac{J(k,\eta)}{J(l,\xi)}-1 \big\vert  &\lesssim \frac {\langle k-l ,\eta-\xi\rangle }{\vert k,l,\xi,\eta \vert^{\frac 23 }} e^{C\rho \vert \eta-\xi,k-l \vert^{\frac13 }}.
\end{align*}
Thus
\begin{align*}
    \vert\eta l-k\xi\vert \big\vert \frac{J(k,\eta)}{J(l,\xi)}-1 \big\vert \lesssim \vert k,\eta\vert^{\frac 1 3 }e^{\tilde C \vert \eta-\xi,k-l \vert^{\frac13 }},
\end{align*}
for some constant $\tilde C>0$.
\noindent
When $t\ge \tfrac 1 2 \min(|\xi|^{\frac 2 3 }, |\eta|^{\frac 2 3 }) $ and $|\eta|\ge 4 |k| $, we know that
\begin{equation*}
    \frac{J(k,\eta)}{J(l ,\xi)}\lesssim |\xi|^\frac12e^{C_0|k-l ,\eta-\xi|^\frac13},
\end{equation*}
for some constant $C_0>0$.
Thus
\begin{align*}
    \vert\eta l-k\xi\vert \big\vert \frac{J(k,\eta)}{J(l,\xi)}-1 \big\vert\lesssim |\xi|^\frac32 e^{\tilde C_0 \vert \eta-\xi,k-l \vert^{\frac13 }}\lesssim |l,\xi|^{\frac{s}{2}}t^{\frac32(\frac 3 2 -\frac{s}{2})}e^{C_1 \vert \eta-\xi,k-l \vert^{\frac13 }},
\end{align*}
for a constant $\tilde C_0, C_1>0$. Combining these estimates, we obtain
\begin{align*}
    \vert\eta l-k\xi\vert \big\vert \frac{J(k,\eta)}{J(l,\xi)}-1 \big\vert&\lesssim (\vert k,\eta\vert^{\frac 1 3 }+\vert l ,\xi \vert^{\frac s 2} t^{\frac 32 (\frac 3 2-\frac s2) } )e^{\tilde C_1\rho \vert \eta-\xi,k-l \vert^{\frac13 }},
\end{align*}
for a suitable constant $\tilde C_1>0.$ Therefore, since $s>1/3$ we obtain
\begin{align}\begin{split}
     T_{G\to \phi,1 }&\lesssim \Vert A \Lambda^{\frac s 2 } \phi\Vert_{L^2 }^2  \Vert A G\Vert_{L^2 }  + t^{\frac 94-\frac {3s}4}\Vert A \Lambda^{\frac s 2 } \phi\Vert_{L^2 }  \Vert A\Lambda^{-1}  G\Vert_{L^2 } \Vert A  \phi\Vert_{L^2 }\\
    &\lesssim \Vert A \Lambda^{\frac s 2 } \phi\Vert_{L^2 }^2  \Vert A G\Vert_{L^2 }  + t^{\frac 54-\frac {3s}4 }\Vert A \Lambda^{\frac s 2 } \phi\Vert_{L^2 }  \Vert A \nabla_t  G\Vert_{L^2 } \Vert A  \phi\Vert_{L^2 },\label{eq:GpT1}
\end{split}\end{align}
which is consistent with \eqref{bd:finTGtophi}.

\medskip
\noindent
\textit{$\diamond$ Bound on $T_{G\to \phi,2}$:} By definition of $m$, if $\tfrac {m(k,\eta)}{m(l,\xi)} \neq 1 $ we obtain
\begin{align*}
    \min (\vert k,\eta \vert, \vert l ,\xi \vert) &\le t^{2}.
\end{align*}
Therefore
\begin{align*}
    \vert \eta l - k \xi\vert  &\le \min(\vert k,\eta\vert ,\vert l,\xi \vert) \vert k-l,\eta -\xi\vert  \\
    &\le t^{2(1-\frac s 2 ) }\min(\vert k,\eta\vert^{\frac s 2 } ,\vert l,\xi \vert^{\frac s 2 }) \vert k-l,\eta -\xi\vert.
\end{align*}
Hence, it is not hard to  deduce that
\begin{align*}
T_{G\to \phi,2 } &\lesssim   t^{2(1-\frac s 2 ) } \Vert A \Lambda^{-1} G \Vert_{L^2} \Vert  A\Lambda^{\frac s 2 }\phi\Vert_{L^2}\Vert  A\phi\Vert_{L^2}\\
    &\lesssim   \langle t \rangle ^{1-s } \Vert  A\Lambda^{\frac s 2 }\phi\Vert_{L^2}\Vert A \nabla_t G \Vert_{L^2}\Vert  A\phi\Vert_{L^2},
\end{align*}
which is consistent with \eqref{bd:finTGtophi} since $1-s\leq \frac54-\frac{3s}{4}$.

\medskip
\noindent
\textit{$\diamond$ Bound on $T_{G\to \phi,3}$: }Here we simply use that
\begin{align*}
    \left\vert e^{\lambda(t) (\vert k,\eta \vert^s-\vert l,\xi  \vert^s)}-1\right\vert &\lesssim  (\vert k,\eta \vert^s-\vert l,\xi  \vert^s) e^{\lambda(t) (\vert k,\eta \vert^s-\vert l,\xi  \vert^s)}\\
    &\le \frac {\vert \eta -\xi, k-l \vert}{\vert k,\eta\vert^{1-s } }e^{c_s \lambda(t) \vert k-l,\eta-\xi \vert^s}.
\end{align*}
Thus
\begin{align}\begin{split}
     T_{G\to \phi,3} &\lesssim \Vert A\Lambda^{\frac s 2 }\phi \Vert^2 \Vert A G \Vert, \label{eq:GpT3}
\end{split}\end{align}
which is in agreement with \eqref{bd:finTGtophi}.

\medskip
\noindent
\textit{$\diamond$ Bound on $T_{G\to \phi,4}$: } Since $\vert l,\xi\vert \approx \vert k,\eta \vert $ we obtain
\begin{align*}
    \vert \tfrac{\langle k,\eta\rangle^N }{\langle l,\xi\rangle^N}-1 \vert\lesssim \vert k- l,\eta- \xi\vert  {\langle l,\xi \rangle^{N-1}}.
\end{align*}
and therefore we infer
\begin{align}
     T_{G\to \phi,4} \lesssim \Vert  A G \Vert_{L^2}\Vert  A \phi  \Vert_{L^2}^2,\label{eq:GpT4}
\end{align}
which is consistent with \eqref{bd:finTGtophi}.

\subsubsection*{$\bullet$ \textbf{Bound on $\calR_{G\to \phi} $}: }In this section we bound the remainder of the action from $G$ to $\phi$
\begin{align*}
     \calR_{G\to \phi}  &\leq
    \sum_{ \substack{k,l \\
    k\neq 0} }\iint \textbf{1}_{S_\calR}\vert \eta l -k \xi\vert  \vert A(k,\eta)-A(l,\xi)\vert\vert A\phi\vert (k,\eta ) \vert G\vert (k-l,\eta-\xi) \vert \phi\vert (l,\xi) d\xi d\eta\\
    &\leq
    \sum_{ \substack{k,l \\
    k\neq 0} }\iint \textbf{1}_{S_\calR}\tfrac {\vert \eta l -k \xi\vert }{\sqrt{k^2+(\eta-kt )^2}} \vert A(k,\eta)-A(l,\xi)\vert\vert A\phi\vert (k,\eta ) \vert \Lambda_t G\vert (k-l,\eta-\xi) \vert \phi\vert (l,\xi) d\xi d\eta.
\end{align*}
We can now proceed exactly as done to control $\mathcal{R}_{\phi \to G}$ and conclude that
\begin{align*}
    \calR_{G\to \phi } \lesssim \langle t \rangle^{-1} \Vert A \nabla_t G \Vert_{L^2} \Vert A\phi\Vert_{L^2}^2,
\end{align*}
as desired for \eqref{bd:fincalRphitoG}.

\subsection{The $NL_\phi$ estimate }
In this section, we want to bound the remaining nonlinear terms arising from the equation of $\phi$
\begin{align*}
    NL_{\phi}&= \langle A\big((\p_x\nabla^\perp \Delta_t^{-1} \phi_{\neq}\cdot \nabla)\phi\big),A\phi\rangle -\langle A(v^x_{0}\p_x \phi_{\neq}),A\phi_{\neq}\rangle .
\end{align*}

For the first term, the nonlinearity is of the form of  \cref{sec:phi} with $G$ replaced by $\p_x\Delta_t^{-1} \phi_{\neq}$. Therefore we can proceed exactly as done to prove \eqref{bd:finRGtophi}-\eqref{bd:fincalRGtophi} and, using that $\|\p_x\Delta_t^{-1} A\phi\|_{L^2}\lesssim \|A\phi\|_{L^2}$ when needed, we get
\begin{align}\begin{split}
    |\langle &A\big((\p_x\nabla^\perp \Delta_t^{-1} \phi_{\neq}\cdot \nabla)\phi\big),A\phi\rangle|\\
    &\lesssim  \langle t\rangle  \Vert A\p_x\Lambda^{-1}_t \phi_{\neq}  \Vert_{L^2}\left(\langle t\rangle^{\frac 1 2}\Vert A\sqrt{\tfrac{\p_t m}{m }} \phi \Vert_{L^2}+\Vert A\phi \Vert_{L^2}\right)\Vert A\phi \Vert_{L^2}\\
    &\quad +\Vert A \Lambda^{\frac s 2 } \phi\Vert_{L^2 }^2  \Vert A\phi_{\neq}\Vert_{L^2 } +\langle t \rangle ^{\frac 54-\frac {3s}4 } \Vert  A\Lambda^{\frac s 2 }\phi\Vert_{L^2}\Vert A  \p_x\Lambda^{-1}_t \phi_{\neq}\Vert_{L^2}\Vert  A\phi\Vert_{L^2}.\label{eq:NLphi1}
\end{split}\end{align}
Hence, in the rest of the section we need to control only the term
\begin{align*}
    &NL_{v_0\to \phi }:=|\langle A(v^x_{0}\p_x \phi_{\neq}),A\phi_{\neq}\rangle| = |\langle [A,v^x_{0}]\p_x \phi_{\neq},A\phi_{\neq}\rangle|\\
    &\le \sum_{k\neq 0 } \iint  (\textbf{1}_{S_R\cup S_\calR }+ \textbf{1}_{S_T }) \vert k \vert \vert A(k,\eta)-A(k,\xi) \vert  \vert A\phi\vert(k,\eta )  \vert v_0^x\vert(\eta-\xi  )  \vert A\phi\vert(k,\xi  ) d\xi d\eta, \\
    &=:R_{v_0\to \phi }+T_{v_0\to \phi }.
\end{align*}
Our next goal is to prove the following bounds
\begin{align}
    \label{bd:Rv0phi}R_{v_0\to \phi }&\lesssim \Vert A \langle \p_y\rangle^{-1}v_0^x \Vert_{L^2}\Vert A \phi  \Vert_{L^2}^2, \\
        \label{eq:boundvp} T_{v_0\to \phi }&\lesssim \big(\Vert  A\phi\Vert^2_{L^2}+\langle t^{1-\frac  3 2 s }\rangle \Vert \Lambda^{\frac s 2 } A\phi\Vert^2_{L^2}\big) \Vert A \langle \p_y \rangle ^{-1}  v_0^x\Vert_{L^2} \\
        \notag &\qquad + t^{\frac 3 2 } \Vert  A\phi\Vert_{L^2}\Vert \sqrt{\tfrac {\p_t q }q}A \phi\Vert_{L^2} \Vert A \langle \p_y \rangle ^{-1} \p_y v_0^x\Vert_{L^2}.
\end{align}
For $R_{v_0\to \phi}$ it is straightforward to check that, thanks to \eqref{bootstrap} and $\eps t^{\frac 3 2 } \le \delta$, one can prove a bound in agreement with \eqref{bd:NLphi}. For $T_{v_0\to \phi}$, it is enough to observe that since $s>1/3$ and $\eps\leq \delta t^{-\frac32}$ we have
\begin{equation*}
\langle t^{1-\frac32s}\rangle \Vert \Lambda^{\frac s 2 } A\phi\Vert^2_{L^2}\Vert A \langle \p_y \rangle ^{-1}  v_0^x\Vert_{L^2}\lesssim \delta \frac{1}{\langle t\rangle^{\frac12+\frac32 s}}\Vert \Lambda^{\frac s 2 } A\phi\Vert^2_{L^2}\lesssim \delta \mathcal{D}_{G,\phi}.
\end{equation*}
Hence, overall we get that
\begin{equation*}
    T_{v_0\to \phi}\lesssim \eps^3+\delta (\mathcal{D}_{G,\phi}+\sqrt{\calD_{G,\phi}}\sqrt{\calD_{v_0}})
\end{equation*}
and integrating in time we prove a bound consistent with \eqref{bd:NLphi}.
\subsubsection*{$\bullet$ \textbf{Bound on $R_{v_0\to \phi }$: }} In the support of the integral we have $\vert \eta-\xi \vert \ge \tfrac 1 8 \vert k , \xi\vert$. Thus, using Lemma \ref{lem:fullJ} v) we infer that
\begin{align*}
    \vert k \vert \vert A(k,\eta)-A(k,\xi) \vert\lesssim \frac 1 {\vert k,  \xi \vert^2\langle \eta-\xi\rangle} A(0,\eta -\xi)A(k,\xi).
\end{align*}
This readily implies the bound \eqref{bd:Rv0phi}.
\subsubsection*{$\bullet$ \textbf{Bound on $T_{v_0\to \phi }$: }}
We first split
\begin{align*}
    A(k,\eta)-A(k,\xi)=\, & A(k,\xi )\left(e^{\lambda(t) (\vert k,\eta \vert^s-\vert k,\xi  \vert^s)}-1  \right)\\
    &+A(k,\xi )e^{\lambda(t) (\vert k,\eta \vert^s-\vert k,\xi  \vert^s)}\big(\frac{\langle k,\eta\rangle^N }{\langle k,\xi\rangle^N}-1  \big)\\
    &+A(k,\xi )e^{\lambda(t) (\vert k,\eta \vert^s-\vert k,\xi  \vert^s)}\frac{\langle k,\eta\rangle^N }{\langle k,\xi\rangle^N}\big(\frac{m(k,\eta )  }{m(k,\xi) }-1  \big)\\
    &+A(k,\xi )e^{\lambda(t) (\vert k,\eta \vert^s-\vert k,\xi  \vert^s)}\frac{\langle k,\eta\rangle^N }{\langle k,\xi\rangle^N}\frac {m(k,\eta )}{m(k,\xi)}\big(\frac{J(k,\eta)}{J(k,\xi)}-1  \big).
\end{align*}
Then, in account of  Lemma \ref{lem:useest} i) we get
\begin{align*}
    T_{v_0\to \phi }
    &\le \sum_{k\neq 0 } \iint  \vert k \vert \bigg(\big\vert \frac{J(k,\eta)}{J(k,\xi)}-1\big \vert  +\big\vert \frac{m(k,\eta )  }{m(k,\xi) }-1  \big\vert+e^{-c_s\lambda(t) \vert \eta -\xi \vert^s}\vert e^{\lambda(t) (\vert k,\eta \vert^s-\vert k,\xi  \vert^s)}-1  \vert\\
    &\qquad  \qquad +\big\vert \frac{\langle k,\eta\rangle^N }{\langle k,\xi\rangle^N}-1  \big\vert\bigg)  \vert A\phi\vert(k,\eta )  \vert v_0^x\vert(\eta-\xi  )  \vert A\phi\vert(k,\xi  ) d\xi d\eta\\
    &= T_{v_0\to \phi,1 }+T_{v_0\to \phi,2 }+T_{v_0\to \phi,3 }+T_{v_0\to \phi,4 }.
\end{align*}
For $T_{v_0\to \phi,3 }$ and $T_{v_0\to \phi,4 }$ we can apply the same steps as done for $T_{G\to \phi,3 }$ and $T_{G\to \phi,4 }$ to get
\begin{align*}
    T_{v_0\to \phi,3 }&\lesssim  \Vert \Lambda^{\frac s 2 } \phi\Vert^2_{L^2} \Vert A\langle \p_y \rangle^{-1}v_0^x\Vert_{L^2}, \\
    T_{v_0\to \phi,4 }&\lesssim  \Vert A\phi\Vert^2_{L^2}\Vert  A\langle \p_y \rangle^{-1}v_0^x\Vert_{L^2},
\end{align*}
which are bounds consistent with \eqref{eq:boundvp}. We control the remaining terms below.

\medskip
\noindent
\textit{$\diamond$ Bound on $T_{v_0\to \phi,1 }$:} For $(t\le \tfrac 1 2 \min(|\xi|^{\frac 2 3 }, |\eta|^{\frac 2 3 }) )$ or ($t\ge \tfrac 1 2 \min(|\xi|^{\frac 2 3 }, |\eta|^{\frac 2 3 }) $ and $|\eta|\le 4|k| $), using  Lemma \ref{lem:fullJ} vi) or vii) respectively we get
\begin{align*}
   \big\vert \frac{J(k,\eta)}{J(k,\xi)}-1 \big\vert  &\lesssim \frac {\langle \eta-\xi\rangle }{\vert k,\xi,\eta \vert^{\frac 23 }} e^{C \vert \eta-\xi\vert^{\frac13 }},
\end{align*}
for some constant $C>0$ that can change from line to line. Thus, since $s>1/3$, we get
\begin{align*}
    \vert k\vert \big\vert \frac{J(k,\eta)}{J(k,\xi)}-1 \big\vert \lesssim \vert k\vert^{\frac 1 3 }e^{C \vert \eta-\xi\vert^{\frac13 }}\lesssim |k,\eta|^se^{C\vert \eta-\xi\vert^{\frac13 }}.
\end{align*}
When  $t\geq \tfrac12\min(|\xi|^\frac23,|\eta|^\frac23)$ and $4\vert k \vert \le  |\eta|$, we can directly apply  Lemma \ref{lem:fullJ} viii).
Therefore, in all cases we get the following bound
\begin{align*}
    \vert k\vert \big\vert \frac{J(k,\eta)}{J(l,\xi)}-1 \big\vert &\lesssim \big(|k,\eta|^{s}+\mathbf{1}_{\{1\leq |k|\lesssim t^\frac32\}}\big(\vert k\vert^{\frac 2 3}+\vert k\vert |\eta-\xi| \sqrt{\tfrac {\p_t q }q (k,\eta )}\big)\big)e^{C|\eta-\xi|^\frac13}\\
    &\lesssim \big(\vert k,\eta \vert^st^{1-\frac  3 2 s }  +t^{\frac 3 2 } |\eta-\xi| \sqrt{\tfrac {\p_t q }q(k,\eta ) }\big)e^{C|\eta-\xi|^\frac13}.
\end{align*}
Thus, it is not hard to conclude that
\begin{align*}
    T_{v_0\to \phi,1 }&\le t^{1-\frac  3 2 s }\Vert \Lambda^{\frac s 2 } A\phi\Vert^2_{L^2} \Vert  A \langle \p_y \rangle ^{-1}v_0^x\Vert_{L^2}+ t^{\frac 3 2 } \Vert  A\phi\Vert_{L^2}\Vert \sqrt{\tfrac {\p_t q }q } A\phi\Vert_{L^2} \Vert A \langle \p_y \rangle ^{-1} \p_y v_0^x\Vert_{L^2}.
\end{align*}
which is agreement with \eqref{eq:boundvp}.

\medskip
\noindent
\textit{$\diamond$ Bound on $T_{v_0\to \phi,2 }$:}  Applying \cref{lem:fullm1} iv) and using \cref{lem:fullJ} v), we can easily infer that
\begin{align*}
    T_{v_0\to \phi,2 }&\lesssim \Vert  A\phi\Vert^2_{L^2}\Vert A\langle \p_y \rangle^{-1} v_0^x\Vert_{L^2}
\end{align*}
thus concluding the proof of \eqref{eq:boundvp}.

\subsection{ The $NL_G$ Estimate }
\label{sec:NLG}
To conclude the bounds on the energy functional $\mathcal{E}$, we have to estimate the term $NL_G$ defined in \eqref{def:NLG}.
 We want to estimate
        \begin{align*}
            NL_G&=-\langle A\big( \Delta^{-1}_t  ( v\cdot \nabla_t w)_{\neq}\big), AG\rangle- \langle    \p_x\Delta_t^{-1}(v\cdot\nabla_t  \phi)_{\neq},AG\rangle\\
            &=:NL_{G,1}+NL_{G,2}.
        \end{align*}
For the first term, since $v$ is divergence-free we use the identity $$ v\cdot \nabla_t w=  \nabla_t^\perp\cdot  (v\cdot \nabla_t v)=\nabla_t^\perp\cdot(\nabla_t\cdot(v\otimes v)),$$ so that  integrating by parts we get
\begin{align*}
    NL_{G,1}=
    -\langle  A ( v\otimes  v)_{\neq}, \Delta^{-1}_t (\nabla^\perp_t\otimes \nabla_t )AG\rangle.
\end{align*}
Thanks to the properties of the weights, it is not hard to deduce that
\begin{equation}
     A(k,\eta )\lesssim\sqrt t  \left(\frac 1 {\langle l ,\xi\rangle^2 }+ \frac 1 {\langle k-l ,\eta-\xi\rangle^2 } \right)  A(k-l,\eta-\xi) A(l,\xi). \label{eq:AG}
\end{equation}
Therefore
\begin{align*}
    |NL_{G,1}|&\lesssim \sqrt t  \Vert A G \Vert_{L^2} \Vert Av\Vert_{L^2}^2.
\end{align*}
Then, observing that
\begin{equation}
\label{eq:idvneq}
    |v_{\neq}|\leq |\Lambda_t^{-1} w_{\neq}|\leq|\Lambda_t G|+|\p_x\Lambda_t^{-1}\phi_{\neq}|,
\end{equation}
we conclude
\begin{align*}
|NL_{G,1}|    &\lesssim \sqrt t \Vert A G \Vert_{L^2} \left(\Vert A \nabla_t G\Vert_{L^2}^2+\Vert A\p_x \Lambda_t^{-1} \phi_{\neq} \Vert_{L^2}^2+\Vert Av_0^x \Vert_{L^2}^2\right).
\end{align*}
For the second term we can do the following rough estimate: integrating by parts and using  \eqref{eq:AG} and \eqref{eq:idvneq} we have
\begin{align*}
    |NL_{G,2}|&=   |\langle A    (v \phi)_{\neq},A\p_x\nabla_t \Delta_t^{-1}G\rangle|\\
    &\lesssim \sqrt t  \Vert AG \Vert_{L^2} \Vert Av \Vert_{L^2} \Vert A\phi \Vert_{L^2} \\
    &\lesssim \sqrt t   \Vert A\phi \Vert_{L^2} \left(\Vert A \nabla_t G\Vert_{L^2}^2+\Vert A\p_x \Lambda_t^{-1} \phi_{\neq} \Vert_{L^2}^2+\Vert A G \Vert_{L^2} \Vert Av_0^x \Vert_{L^2}\right).
\end{align*}
Appealing to  the bootstrap hypotheses \eqref{bootstrap}, we thus get
\begin{equation*}
    |NL_G|\lesssim \sqrt{t}\eps \mathcal{D}_{G,\phi}+\sqrt{t}\eps^3.
 \end{equation*}
Hence, integrating in time we infer that for some constant $C>0$ we get
\begin{align*}
    \int_0^T \vert NL_G\vert d\tau  \lesssim\sqrt T\eps^3+T^\frac32 \eps^3\leq C \delta\eps^2 ,
\end{align*}
where the last inequality follows since $T\leq \delta \eps^{-\frac23}$.

\subsection{The $NL_{v_0}$ estimate } In this section, we bound the nonlinear term for $v_0^x$ defined in \eqref{def:NLv0}. We first observe that
         \begin{align*}
            NL_{v_0}=  \langle A\langle \p_y\rangle^{-1}( b_{\neq}^y \p_y^t b_{\neq}^x -v_{\neq}^y \p_y^t v^x_{\neq})_0, A\langle \p_y\rangle^{-1}v_0^x
            \rangle.
        \end{align*}
        Then, since $\p_y^tb_{\neq}^y=-\p_xb_{\neq}^x$ and $\p_y^tv_{\neq}^y=-\p_xv_{\neq}^x$, integrating by parts the identity above, recalling that $b=\nabla^\perp_t \phi$ and $v=\nabla^\perp_t \psi$, we obtain
 \begin{align*}
            NL_{v_0}&= \langle A\langle \p_y\rangle^{-1} (\p_x\phi_{\neq} \p_y^t\phi_{\neq})_0, A\langle \p_y\rangle^{-1}\p_yv_0^x
            \rangle\\
            &\quad - \langle A\langle \p_y\rangle^{-1} ( \p_x\psi_{\neq}  \p_y^t\psi_{\neq})_0, A\langle \p_y\rangle^{-1} \p_yv_0^x
            \rangle \\
            &=:NL_{v_0,1}+NL_{v_0,2}.
        \end{align*}
The most challenging term is the one involving $\phi_{\neq}$. Indeed, for $\psi_{\neq}$ we can use the relation $\psi_{\neq}=G-\p_x\Delta_t^{-1}\phi_{\neq}$ and, since $G$ enjoys dissipative properties, the term involving $\psi_{\neq}$ is easier compared to the one for $\phi_{\neq}$. Therefore, we only present the detailed computations for $NL_{v_0,1}$. A first trivial bound gives
    \begin{align*}
        |NL_{v_0,1}|\lesssim \sum_{k\neq 0}\iint \frac{A(0,\eta)}{\langle \eta \rangle}|\xi-kt||k| |\phi_{\neq}(k,\xi)||\phi_{\neq}(-k,\eta-\xi)|\frac{A(0,\eta)}{\langle \eta \rangle}|\p_yv_0^x|(\eta) d\eta d\xi
    \end{align*}
In account of Lemma \ref{lem:fullJ} and Lemma \ref{lem:useest}, it is not hard to deduce the following inequality
\begin{align*}
    A(0,\eta ) \lesssim\, & \left(\frac 1 {\vert k,\xi\vert^5 }+\frac 1 {\vert k,\eta - \xi\vert^5 }  \right) \\
    &\times \left(1 + \sqrt t \sqrt{\frac {\p_t q } q (-k,\eta-\xi) }+ \sqrt t \sqrt{\frac {\p_t q } q (k,\xi) }\right) \tilde A(-k,\eta -\xi)\tilde A(k,\xi).
\end{align*}
Notice that the gain of one derivative in $y$ is crucial to bound $$\frac{|\xi-kt|}{\langle\eta\rangle}\lesssim  t\min\{|k,\xi|,|k,\eta-\xi|\}$$ and thus we never have a dangerous derivative loss.
Hence,
\begin{align*}
    |NL_{v_0,1}|\lesssim t\|A\langle \p_y\rangle^{-1}\p_yv_0^x\|_{L^2}\|A\phi\|_{L^2}(\|A\phi\|_{L^2}+t^\frac12\|\sqrt{\tfrac{\p_t q}{q}}\tilde A \phi\|_{L^2}).
\end{align*}
In view of the bootstrap hypothesis \eqref{bootstrap}, we get
\begin{equation*}
    |NL_{v_0,1}|\lesssim t^\frac32 \eps \sqrt{\calD_0}\sqrt{\calD_{G,\phi}}+t\eps^2\sqrt{\calD_0}.
\end{equation*}
Integrating the last inequality in time and using H\"older's inequality we then deduce that
\begin{equation*}
    \int_0^T |NL_{v_0,1}| d\tau \lesssim T^\frac32 \eps^3\leq C\delta \eps^2,
\end{equation*}
for a suitable constant $C>0.$
\section{Nonlinear Instability: Long-Time Growth of the Current}
\label{sec:Instability}
In this section, we establish the instability of current and magnetic field as stated in Theorem \ref{thm:main}.
More precisely, we show that for initial data with suitable Fourier support, the linearized dynamics exhibit current growth and that this behavior is stable under small nonlinear perturbations.

\begin{lemma}\label{lem:lowbd}
    Consider a solution which satisfies the bootstrap estimates \eqref{bootstrap} for $0\le t\le \delta \eps^{-\frac23}$. Let $K,c >0$ and $\chi (k) := \textbf{1}_{\vert k \vert \ge 20e^{2\pi}} $. If the initial data additionally satisfies
    \begin{align*}
        \Vert \chi  \phi_{in} \Vert_{H^{-2}}\ge \max (K \delta \eps, \Vert \chi\p_x  G \Vert_{H^{-2}} , c\Vert   \phi_{in} \Vert_{H^{-2}}),
    \end{align*}
     then for all times $0\leq t\le \delta \eps^{-\frac23}$ it holds that
\begin{align*}
    \Vert j(t)\Vert_{L^2}&\gtrsim t^2 \Vert   \phi_{in} \Vert_{H^{-2}},\\
    \Vert b(t)\Vert_{L^2}&\gtrsim t \Vert   \phi_{in} \Vert_{H^{-2}} .
\end{align*}
\end{lemma}

\begin{proof}
We write the nonlinear solutions $\phi, \p_x G$ as a perturbation around the linear solution:
\begin{align*}
    \p_t \begin{pmatrix}
        \phi_{lin}\\ \p_x G_{lin}
    \end{pmatrix}&=L\begin{pmatrix}
        \phi_{lin}\\ \p_x G_{lin}
    \end{pmatrix},\\
    \p_t \begin{pmatrix}
        \phi-\phi_{lin}\\ \p_x (G-G_{lin})
    \end{pmatrix}&=L\begin{pmatrix}
        \phi-\phi_{lin}\\ \p_x (G- G_{lin})
    \end{pmatrix}+ NL[\phi,G],
\end{align*}
where we introduced the short-hand-notation
\begin{align*}
    L&= \begin{pmatrix}
        - \p_x^2 \Delta_t^{-1} & 1 \\
        - \p_x^4 \Delta_t^{-2} &(\Delta_t+ 2 \p_x \p_y^t\Delta_t^{-1} + \p_x^2 \Delta_t^{-1})
    \end{pmatrix},\\
    NL[\phi,G]&= \begin{pmatrix}
        - \p_x^{-1}( \nabla^\perp G\cdot \nabla \phi) + \p_x^{-1}((\p_x\nabla^\perp \Delta_t^{-1} \phi_{\neq}\cdot \nabla) \phi) - \p_x^{-1}(v^x_{0}\p_x \phi_{\neq}) \\
        \Delta^{-1}_t \nabla^\perp_t \p_x \cdot(( \nabla^\perp \phi  \cdot\nabla) \nabla^\perp_t  \phi )_{\neq}-\Delta^{-1}_t  \p_x ( v\cdot \nabla_t w)_{\neq} -   \p_x^2\Delta_t^{-1}(v\cdot\nabla_t  \phi)_{\neq}
    \end{pmatrix}.
\end{align*}
 Let $k_0 = 20e^{2\pi}$, and recall the multiplier $\chi$ given by
\begin{align*}
    \chi(k) &= \begin{cases}
        1 & \vert k \vert \ge k_0,\\
        0 & \text{else}.
    \end{cases}
\end{align*}
Then for $f,g\in L^2$, we define the inner-semi-product
\begin{align*}
    \langle f,g \rangle_X := \langle \chi \Lambda^{-2} f,\chi \Lambda^{-2} g \rangle_{L^2 }
\end{align*}
and $\Vert \cdot \Vert_X$ the corresponding semi-norm. This semi-norm satisfies $\Vert \cdot \Vert_X\le \Vert \cdot \Vert_{L^2}$.

We consider initial data, which satisfies $\Vert \phi_{in} \Vert_{X}\ge \Vert  \p_x G _{in} \Vert_{X} $ and claim that the following estimates hold
\begin{align}
    C^{-1} \Vert \phi_{in} \Vert_{X} \le \Vert \phi_{lin}  \Vert_{X}&\le C \Vert \phi_{in} \Vert_{X},\label{eq:lowlin}\\
    \Vert \phi- \phi_{lin}  \Vert_{X} &\le  C \delta \eps^{\frac 4 3 }. \label{eq:lowupNL}
\end{align}
for some $C>0$. First we assume the estimates \eqref{eq:lowlin} and \eqref{eq:lowupNL} and prove Lemma \ref{lem:lowbd} as follows: For $\Vert \phi_{in} \Vert_{X}\ge K \eps \delta$ with $ K \ge  2C^2 $ we deduce, that
\begin{align*}
    \Vert \phi \Vert_X &\ge \Vert \phi_{lin}  \Vert_{X}-  \Vert \phi- \phi_{lin}  \Vert_{X} \\
    &\ge C^{-1} \Vert \phi_{in} \Vert_{X}- C \delta \eps^{\frac 4 3 }\\
    &= \frac 1 2 C^{-1}\Vert \phi_{in} \Vert_{X}+\delta \eps  K\left( \frac 1 2C^{-1} -  K^{-1} C\right).
\end{align*}
Therefore,
\begin{align*}
    \Vert b \Vert_{L^2} &= \Vert \nabla^\perp_t  \phi \Vert_{L^2} = \Vert \Lambda_t \Lambda \phi \Vert_{H^{-1} }\ge t \Vert \chi \phi \Vert_{H^{-2} }= t \Vert \chi \phi_{in} \Vert_{X},\\
    \Vert j \Vert_{L^2} &= \Vert \Delta_t   \phi \Vert_{L^2} = \Vert \Delta_t \Delta\phi \Vert_{H^{-2} }\ge t^2 \Vert \phi \Vert_{H^{-2} }= t^2\Vert \phi_{in} \Vert_{X}.
\end{align*}
Thus, since $\Vert \chi \phi_{in} \Vert_{X}\ge c \Vert \phi_{in} \Vert_{X}$, Lemma \ref{lem:lowbd} holds if \eqref{eq:lowlin} and \eqref{eq:lowupNL} hold. So it is only left to prove \eqref{eq:lowlin} and \eqref{eq:lowupNL}, we first prove \eqref{eq:lowlin}. We estimate
\begin{align*}
    \frac 1 2 \p_t \Vert (\phi_{lin} , \p_x G_{lin}) \Vert^2_{X}&=- \Vert \p_x \Lambda_t^{-1} \phi_{lin}\Vert_{X}+ \langle \p_x G ,(\Delta_t+ 2 \p_x \p_y^t\Delta_t^{-1} + \p_x^2 \Delta_t^{-1})\p_x G\rangle_X  \\
    &- \langle \p_x G   , \p_x^4 \Delta_t^{-2} \phi \rangle_X + \langle \phi, \p_x G \rangle_X.
\end{align*}
It holds that
\begin{align*}
    \vert \langle \phi, \p_x G \rangle_X\vert &\le \vert \langle \p_x \Lambda_t^{-1}\phi , \p_x^{-1} \p_x \Lambda_t G \rangle_X\vert  \le \frac 1{2k_0}( \Vert \p_x \Lambda_t^{-1}\phi\Vert_X^2 +\Vert  \p_x \Lambda_t G\Vert_X^2)\\
    \vert \langle \p_x G   , \p_x^4 \Delta_t^{-2} \phi \rangle\vert &\le \frac 1{2k_0}( \Vert \p_x \Lambda_t^{-1}\phi\Vert_X^2 +\Vert  \p_x \Lambda_t G\Vert_X^2)
\end{align*}
and
 \begin{align*}
     \langle \p_x G ,(\Delta_t+ 2 \p_x \p_y^t\Delta_t^{-1} + \p_x^2 \Delta_t^{-1})\p_x G\rangle&\le\left(\frac 2 {k_0} -1 \right)  \Vert \nabla_t \p_x G \Vert_X.
 \end{align*}
Therefore, since ${k_0}\ge 10$, we obtain
\begin{align*}
     \p_t \Vert (\phi_{lin} , \p_x G_{lin}) \Vert^2_{X}+ \frac 1 2\Vert \p_x \Lambda_t^{-1} \phi_{lin}\Vert_{L^2} + \frac 1 2 \Vert \nabla_t \p_x G_{lin} \Vert_{X} &\le 0
\end{align*}
and so
\begin{align}
    \Vert (\phi_{lin} , \p_x G_{lin}) \Vert^2_{X}(t)+ \frac 12 \Vert \p_x \Lambda_t^{-1} \phi_{lin}\Vert_{L^2X} + \frac 1 2 \Vert \nabla_t \p_x G_{lin} \Vert_{L^2X}&\le \Vert \phi_{in} , \p_x G_{in}\Vert_{X},\label{eq:lowup}
\end{align}
which proves the upper bound of \eqref{eq:lowlin}. For the lower bound we define the multiplier
\begin{align*}
    \tilde m&= \exp\left( \int_0^t \frac 1 {1+ (\tau-\frac \eta k )^2}d\tau\right),
\end{align*}
then we obtain
\begin{align*}
    \p_t \Vert \tilde  m \phi\Vert_{X}^2&= 2\langle \tilde m\phi,\tilde  m\p_x G\rangle_X \ge -\frac {2e^{2\pi} } {\vert k_{0}\vert }\Vert \p_x \Lambda_t^{-1} \phi_{lin}\Vert_{X}  \Vert \nabla_t \p_x  G_{lin} \Vert_{X}.
\end{align*}
Integrating in time yields
\begin{align*}
    \Vert \tilde m \phi\Vert^2_{X }(t) &\ge \Vert \phi_{in}\Vert^2_{ X }-\frac {2e^{2\pi} } {\vert k_{0}\vert }\Vert \p_x \Lambda_t^{-1} \phi_{lin}\Vert_{L^2 X }  \Vert \nabla_t  \p_x G_{lin} \Vert_{L^2 X }\\
    &\le \Vert \phi_{in}\Vert-\frac {8e^{2\pi} } {\vert k_{0}\vert }\Vert \phi_{in},  \p_x G_{in}\Vert^2_X\\
    &\ge \left(1-\frac {16e^{2\pi} } {\vert k_{0}\vert }\right) \Vert \phi_{in}\Vert^2_X,\\
    &\gtrsim \Vert \phi_{in}\Vert^2_X
\end{align*}
Therefore, we infer
\begin{align*}
    \Vert  m \phi\Vert^2_{X }(t) &\gtrsim   \Vert \phi_{in}\Vert^2_X.
\end{align*}
which proves the lower bound \eqref{eq:lowlin}.

Now we turn to \eqref{eq:lowupNL}, we want to use the bounds claimed in \eqref{bootstrap}, which we proved to hold in the previous section, to estimate the nonlinear contributions
\begin{align*}
    \frac 12 \p_t \left\Vert \begin{pmatrix}
        \phi-\phi_{lin} \\ \p_x (G-G_{lin})
    \end{pmatrix}  \right\Vert^2_X &= \left\langle \begin{pmatrix}
        \phi-\phi_{lin} \\ \p_x (G-G_{lin})
    \end{pmatrix} , L \begin{pmatrix}
        \phi-\phi_{lin} \\ \p_x (G-G_{lin})
    \end{pmatrix}  + NL[\phi, G ] \right\rangle.
\end{align*}
The linear estimates
\begin{align*}
    \left \langle \begin{pmatrix}
        \phi-\phi_{lin} \\ \p_x (G-G_{lin})
    \end{pmatrix} , L \begin{pmatrix}
        \phi-\phi_{lin} \\ \p_x (G-G_{lin})
    \end{pmatrix}  \right\rangle&\le -\frac 1 2 \Vert \p_x \Lambda_t^{-1} (\phi- \phi_{lin})\Vert_{X} - \frac 1 2\Vert \nabla_t \p_x (G-G_{lin}) \Vert_{X},
\end{align*}
are the same as the ones obtained for $\phi_{lin},G_{lin}$. For the nonlinear terms, we can easily handle regularity losses in account of the high regularity in which we proved \eqref{bootstrap}. We show only the bound for the nonlinear contribution of $\p_x \Delta^{-1}_t \p_y^t (( \nabla^\perp \phi  \cdot\nabla) \p_y^t  \phi )_{\neq}$, all the other nonlinear terms are simpler to control. We estimate
\begin{align*}
    \vert \langle \p_x (G&-G_{lin}) , \p_x \Delta^{-1}_t \p_y^t (( \nabla^\perp \phi  \cdot\nabla) \p_y^t  \phi )_{\neq} \rangle_X\vert  \\
    &\le  \sum_{\substack{k,l\\ k\neq 0 }} \iint \chi(k)\frac {\vert k\vert^2}{\vert k,\eta \vert^2} \frac {\vert (\eta l -k\xi)(\eta -kt)(\xi-lt)\vert }{(k-l)^2(1+ (\frac {\eta-\xi}{k-l} -t )^2 }\vert G-G_{lin}\vert (k,\eta)\vert \\
    &\hspace{2cm}\times \phi\vert (k-l,\eta-\xi)\vert \phi\vert (l,\xi)  d\xi d\eta .
\end{align*}
Using $\tfrac {\vert (\eta-kt)  (\xi-lt)\vert }{(k-l)^2(1+ (\frac {\eta}{k} -t )^2 }\le  \langle  l,\xi\rangle^2+ \langle  k-l,\eta-\xi\rangle^2 $ we obtain
\begin{align*}
   \vert  \langle \p_x (G- G_{lin}) &, \p_x \Delta^{-1}_t \p_y^t \cdot(( \nabla^\perp \phi  \cdot\nabla) \p_y^t  \phi )_{\neq} \rangle_X \vert  \lesssim \Vert A (G- G_{lin} ) \Vert_{L^2}\Vert A \phi\Vert_{L^2}^2.
\end{align*}
The other nonlinear terms are estimated in a similar manner, thus we obtain
\begin{align*}
    &\left\vert  \left\langle \begin{pmatrix}
        \phi- \phi_{lin} \\ \p_x (G-G_{lin})
    \end{pmatrix} ,  NL[\phi, G ] \right\rangle\right\vert \\
    &\quad \lesssim \left\Vert \begin{pmatrix}
        \phi- \phi_{lin} \\ \p_x (G-G_{lin})
    \end{pmatrix}\right\Vert_{L^2}  (\Vert AG \Vert_{L^2}^2 + \Vert A \phi\Vert_{L^2}^2+ \Vert \frac {A }{\langle \p_y \rangle }v_0^x \Vert_{L^2}^2)
\end{align*}
Therefore, after integrating in time we infer
\begin{align*}
    &\left\Vert \begin{pmatrix}
        \phi- \phi_{lin} \\ \p_x (G-G_{lin})
    \end{pmatrix}\right\Vert_{L^\infty L^2}^2 \\
    &\quad \le C t  \left\Vert \begin{pmatrix}
        \phi- \phi_{lin} \\ \p_x (G-G_{lin})
    \end{pmatrix}\right\Vert_{L^\infty L^2} (\Vert AG \Vert_{L^\infty  L^2}^2 + \Vert A \phi\Vert_{L^\infty L^2}^2+ \Vert \frac {A }{\langle \p_y \rangle }v_0^x \Vert_{L^\infty L^2}^2).
\end{align*}
Appealing to the bounds in \eqref{bootstrap}, we obtain
\begin{align*}
    \left\Vert \begin{pmatrix}
        \phi- \phi_{lin} \\ \p_x (G-G_{lin})
    \end{pmatrix}\right\Vert_{L^\infty L^2}&\le C t  \eps^2 \le C \delta \eps^{\frac 4 3 }.
\end{align*}
Thus, estimate \eqref{eq:lowupNL} holds and so we deduce Lemma \ref{lem:lowbd}.

\end{proof}

\subsection*{Acknowledgments}
N. Knobel and Ch. Zillinger have been funded by the Deutsche Forschungsgemeinschaft (DFG, German Research Foundation) – Project-ID 258734477 – SFB 1173.

N. Knobel is supported by ERC/EPSRC Horizon Europe Guarantee EP/X020886/1 while working at Imperial College London.

M. Dolce is supported by the Swiss State Secretariat for
Education, Research and Innovation (SERI) under contract number MB22.00034 through the
project TENSE and is a member of the GNAMPA-INdAM.

\appendix

\section{Fourier Weights}

\subsection{Proof of Lemma \ref{lem:fullm1}}

\begin{proof}[Proof of Lemma \ref{lem:fullm1}]
We start with the proof of i). The lower bound is clear since $m(0,k,\eta)=1$ and $\partial_t m\ge 0$. For the upper bound, we estimate
    \begin{align*}
        \frac {\partial_t m}m (t,k,\eta) &\le \sup\limits_{j\in \mathbb{Z}\setminus\{0\}} \left(\frac 1{1+(\tfrac \eta j -t )^2 }\frac 1{\langle k-j \rangle^3 }\right)\\
        &\le \sum_{j\in \mathbb{Z}\setminus\{0\}} \left(\frac 1{1+(\tfrac \eta j -t )^2 }\frac 1{\langle k-j \rangle^3 }\right).
    \end{align*}
    Therefore, since $m(0,k,\eta)=1$ it holds that
    \begin{align*}
        m(k,\eta, t) &\le\exp\left( \sum_{j\in \mathbb{Z}\setminus\{0\}} \int_0^t\frac 1{1+(\tfrac \eta j -\tau )^3 }\frac 1{\langle k-j \rangle^3 }  d\tau \right)\\
        &\le \exp\left( \pi \sum_{j\in \mathbb{Z}\setminus\{0\}}\frac 1{\langle k-j \rangle^2 }  d\tau \right)\le \exp\left( \pi^3/6  \right),
    \end{align*}
and thus i) is proved.

For the property ii), we show  that
\begin{align}
    \frac 1 {1+\vert t-\frac \eta k \vert}\le 2\frac {\vert k-l,\eta-\xi\vert}{1+\vert t-\frac {\xi} k \vert}.\label{eq:m1help}
\end{align}
Indeed, once the inequality above is proved the claim follows by the definition of $m$. When  $2\vert k-l,\eta- \xi\vert \ge 1+\vert t-\frac {\xi} k \vert$ then \eqref{eq:m1help} is clear. If
 $2\vert k-l,\eta-\xi\vert \le 1+\vert t-\frac {\xi} k \vert$ we obtain
\begin{align*}
    1+ \vert t-\tfrac {\xi} k \vert &\le  1+ \vert t-\tfrac {\eta} k\vert + \tfrac {\vert \eta-\xi\vert }k \\
    &\le 1+ \vert t-\tfrac {\eta} k\vert+\tfrac 1 2 (1+ \vert t-\tfrac {\xi} k \vert)
\end{align*}
and thus \eqref{eq:m1help} follows by rearranging the terms.

Turning our attention to iii), we see that the claim is clear for $\vert k\vert \ge \vert\eta \vert $. For  $\vert k\vert \le \vert\eta \vert$ and  $\langle t\rangle ^2\ge \vert k , \eta  \vert $, using ii) we get
\begin{align*}
    \frac { \tfrac {|\eta|} {k^2} } {1+(t-\frac \eta k )^2}&\leq \big(\frac{| \eta|} {|k|^3}\big)^{\frac 1 2 } \frac { \vert\tfrac \eta k-t\vert^{\frac 1 2} +t^{\frac 1 2 } }  {1+(t-\frac \eta k )^2}\\
    &\le \big(\frac {|\eta|} {|k|^3}\big)^{\frac 1 2 }  \frac 1 {1+\vert \frac \eta k -t \vert} \left(1+\sqrt { t}\sqrt{\tfrac {\p_t m }m  (t,l,\xi )}\right)\langle k-l,\eta-\xi\rangle^3.
\end{align*}
When $\langle t\rangle ^2\le \vert k,\eta  \vert $, we obtain
        \begin{align*}
        \frac {|\eta|}{k^2}        &\le 1+ \frac {\eta^2 }{k^2}\frac 1{\vert k,\eta\vert }\lesssim 1+ \frac 4{\vert k,\eta\vert }(t^2+ (t-\tfrac \eta k )^2)\lesssim  1+\frac 1 {\langle t \rangle^2 } (t-\tfrac \eta k )^2,
    \end{align*}
    whence proving the desired result.

    Finally, to prove property iv), we apply the mean value theorem for Lipschitz functions.
If we assume that
\begin{align}
    \left\vert\frac {\p_\eta m} m (t,k,\eta) \right\vert &\lesssim   \frac 1 {|k|}\label{eq:mkest}
\end{align}
holds in the Lipschitz sense we obtain that for some $\xi'\in(\eta,\xi)$
\begin{align*}
    \vert m(t,k,\eta ) - m(t,k,\xi)\vert &\le \vert \p_\eta m(t, k,\xi')\vert \vert \eta-\xi\vert  \lesssim  \frac {\vert \eta-\xi\vert}{|k|} .
\end{align*}
So we only need to prove \eqref{eq:mkest} and the $\p_\eta$ is in the Lipschitz sense. The weight $m$ is described by
\begin{align*}
    m(t,k,\eta) &= \begin{cases}
     \exp\left(\int_{\sqrt{\vert k,\eta\vert}}^t \sup\limits_{j\in \mathbb{Z}\setminus\{0\}} \left(\frac 1{1+(\frac {\eta} j -\tau )^2 }\frac 1{\langle k-j \rangle^3 }\right) d\tau\right)&\sqrt{\vert k,\eta\vert}<t\\
     1& \text{else}
     \end{cases}
\end{align*}
and thus we only need to consider $\sqrt{\vert k,\eta\vert}<t$. By the Leibnitz integral rule, we obtain
\begin{align*}
    \frac {\p_\eta m} m (t,k,\eta) &= \p_{\eta}  \int_{\sqrt{\vert k,\eta\vert}}^t \sup\limits_{j\in \mathbb{Z}\setminus\{0\}} \left(\frac 1{1+(\frac {\eta} j -\tau )^2 }\frac 1{\langle k-j \rangle^3 }\right) d\tau\\
    &=   \int_{\sqrt{\vert k,\eta\vert}}^t \p_{\eta} \sup\limits_{j\in \mathbb{Z}\setminus\{0\}} \left(\frac 1{1+(\frac {\eta} j -\tau )^2 }\frac 1{\langle k-j \rangle^3 }\right) d\tau \\
    &- \frac \eta {2\vert k,\eta\vert^{\frac 3 2 }}\sup\limits_{j\in \mathbb{Z}\setminus\{0\}} \left(\frac 1{1+(\frac {\eta} j -t )^2 }\frac 1{\langle k-j \rangle^3 }\right).
\end{align*}
Therefore, we infer
\begin{align*}
    \left\vert \frac {\p_\eta m} m (t,k,\eta)\right\vert  &\le   \int_{\sqrt{\vert k,\eta\vert}}^t\left\vert  \p_{\eta} \sup\limits_{j\in \mathbb{Z}\setminus\{0\}} \left(\frac 1{1+(\frac {\eta} j -\tau )^2 }\frac 1{\langle k-j \rangle^3 }\right)\right\vert  d\tau \\
    &\qquad + \frac 1 {2\vert k,\eta\vert^{\frac 1 2 }}\sup\limits_{j\in \mathbb{Z}\setminus\{0\}} \left(\frac 1{1+(\frac {\eta} j -t )^2 }\frac 1{\langle k-j \rangle^3 }\right)\\
    &=:M_1 +M_2 .
\end{align*}
First we estimate $M_1$, since $\frac 1{1+(\frac {\eta} j -\tau )^2 }\frac 1{\langle k-j \rangle^3 }$ is positive we infer
\begin{align*}
    \left\vert \p_{\eta}\sup\limits_{j\in \mathbb{Z}\setminus\{0\}} \left(\frac 1{1+(\frac {\eta} j -\tau )^2 }\frac 1{\langle k-j \rangle^3 }\right)\right\vert &\le \sup\limits_{j\in \mathbb{Z}\setminus\{0\}} \left\vert\p_{\eta}  \frac 1{1+(\frac {\eta} j -\tau )^2 }\frac 1{\langle k-j \rangle^3 }\right\vert\\
    &\le \sup\limits_{j\in \mathbb{Z}\setminus\{0\}} \left\vert \frac {\frac 2 j (\frac {\eta} j -\tau)}{(1+(\frac {\eta} j -\tau )^2)^2 }\frac 1{\langle k-j \rangle^3 }\right\vert\\
    &\le \frac 2 {|k|} \sum_j \frac {1}{(1+(\frac {\eta} j -\tau )^2)^2 }\frac 1{\langle k-j \rangle^2 }.
\end{align*}
Thus,
\begin{align*}
    \vert M_1\vert &\lesssim \frac 1 {|k|}\int \sum_j \frac {1}{(1+(\frac {\eta} j -\tau )^2)^2 }\frac 1{\langle k-j \rangle^2 }d\tau \lesssim \frac 1 {|k|}.
\end{align*}
For the term $M_2$, we get a bound consistent with \eqref{eq:mkest} if we show the following: for $\sqrt{\vert k,\eta\vert}<t$ and  $j \neq 0$ it holds the estimate
\begin{align}
   \frac 1 {\sqrt{\vert k,\eta\vert}}  \frac 1{1+(\frac {\eta} j -t )^2 }\frac 1{\langle k-j \rangle^3 }\lesssim   \frac 1 {|k|} .\label{eq:M2est}
\end{align}
For $\vert k-j\vert \ge \frac 1 {10} |k|$, $\frac {\eta} j\le \frac t 2 $ or $\frac {\eta} j\ge 2t $ the inequality \eqref{eq:M2est} is clear. For $\vert k-j\vert \le \frac 1 {10} |k|$ and $\frac t 2 \le \frac {\eta} j\le 2t $ we obtain $\frac {\eta}k \ge \frac 1 3t$. Hence, since $\sqrt{\vert k,\eta\vert}< t\le 3 \frac {\eta} k $, we know that $|\eta|\gtrsim k^2 $ which proves \eqref{eq:M2est}. Therefore, \eqref{eq:mkest} holds and  the bound  iv) is proved.

\end{proof}

\subsection{Proof on $q$ and $J$}

\begin{proof}[{Proof of Lemma \ref{lem:fullJ}}]
    The inequalities i)-v) are a consequence of the construction of $J$, $q$ and
    Lemma \ref{lem:fullq}. The inequality vi) is a simple adaption of Lemma 3.9
    in \cite{masmoudi2022stability}. The inequality vii) can be obtained by the
    same steps done in  \cite{NiklasMHD2024}. So it only remains to prove viii).

    By the definition of $J$ in \eqref{def:J}, we see that
    \begin{equation*}
        |J(t,k,\eta)-J(t,k,\xi)|\leq \frac{|e^{8\rho|\eta|^\frac13}-e^{8\rho|\xi|^\frac13}|}{q(t,k,\eta)}+\frac{e^{8\rho|\xi|^\frac13}}{q(t,k,\eta)q(t,k,\xi)}|q(t,k,\eta)-q(t,k,\xi)|.
    \end{equation*}
    For the first term, applying the mean value theorem and Lemma \ref{lem:fullJ} v) we deduce that
    \begin{equation*}
    \frac{|e^{8\rho|\eta|^\frac13}-e^{8\rho|\xi|^\frac13}|}{q(t,k,\eta)}\lesssim \frac{|\eta-\xi|}{\langle \eta\rangle^\frac23} e^{C|\eta-\xi|^\frac13}J(t,k,\xi).
    \end{equation*}
    For the second term, we only need to estimate $|q(t,k,\eta)-q(t,k,\xi)|$. We can consider $t<2\max(|\eta|,|\xi|)$ since otherwise $q(t,k,\eta)=q(t,k,\xi)=1$ and we are done. Moreover, since $t>\min(|\xi|^\frac23,|\eta|^\frac23)$ we only have to consider the case when $q=q_{NR}$ or $q=q_R$. We want to apply again the mean value theorem but we now observe that $\p_\eta q_{NR}$ has jump discontinuities and therefore we need to be careful.

    Once we know how to control $q_{NR}$ then the bounds on $q_R$ are straightforward. Indeed
    \begin{align*}
        \frac{\p_\eta q_R(t,\eta)}{q_R(t,\eta)}=\frac{\p_\eta((\tfrac{k^3}{2\eta}(1+a_{k,\eta}|t-\tfrac{\eta}{k}|))}{\tfrac{k^3}{2\eta}(1+a_{k,\eta}|t-\tfrac{\eta}{k}|)} +\frac{\p_\eta q_{NR}(t,\eta)}{q_{NR}(t,\eta)}.
    \end{align*}
Then, since
\begin{align*}
    \p_\eta((\tfrac{k^3}{2\eta}(1+a_{k,\eta}|t-\tfrac{\eta}{k}|))=-\frac{k^3}{2\eta^2}(1+a_{k,\eta}|t-\tfrac{\eta}{k}|)+\frac{k^6}{\eta^3}|t-\frac{\eta}{k}|+a_{k,\eta}\frac{k^2}{2\eta}\frac{\tfrac{\eta}{k}-t}{|\tfrac{\eta}{k}-t|},
\end{align*}
and $|k|\lesssim |\eta|^\frac13$ we deduce that
\begin{align}
\label{eq:petaqR}
       \big| \frac{\p_\eta q_R(t,\eta)}{q_R(t,\eta)}\big| \lesssim \frac{1}{|k|} +\big| \frac{\p_\eta q_{NR}(t,\eta)}{q_{NR}(t,\eta)}\big| .
    \end{align}
    Hence, we only have to estimate $\p_\eta q_{NR}$.\\
    Let $t\in \tilde I_{j,\eta }$ for some $j$. By definition it holds that
    \[
    q_{NR}(t_{j,\eta}^-,\eta)=\left(\frac{j^3}{2\eta}\right)^{\frac12+ 2\rho}q_{NR}(t_{k,\eta}^+,\eta)=\left(\frac{j^3}{2\eta}\right)^{\frac12+ 2\rho}q_{NR}(t_{j-1,\eta}^-,\eta).
    \]
    Iterating this identity we find that when $t\in I_{k,\eta}^R$ we get
    \begin{equation*}
        q_{NR}(t,\eta)= \left(\frac{j^3}{\eta}\right)^{\rho+\frac12 }(1+a_{j,\eta}|t-\tfrac{\eta}{j}|)^{\rho+\frac12}\prod_{l=1}^{j-1}\left(\frac{l^3}{\eta}\right)^\frac12.
    \end{equation*}
    Then,
    since $|j|\lesssim |\eta|^\frac13$, we observe that
    \begin{align*}
        |\p_\eta(\eta^{-\frac{j}{2}-\rho})|&=\big|\frac{j/2+\rho}{\eta}\big||\eta^{-\frac{j}{2}-\rho}|\lesssim \frac{1}{|\eta|^\frac23}|\eta^{-\frac{j}{2}-\rho}|,\\
        |\p_\eta( a_{j,\eta} )|&=2\frac {|j|^3} {|\eta|^3}\lesssim \frac 1 {|\eta|^2}.
    \end{align*}
    Therefore,
\begin{equation}
\label{bd:deetaqNR}
    \big|\frac{\p_\eta q_{NR}(t,\eta)}{q_{NR}(t,\eta)}\big|\lesssim \frac{1}{|\eta|^{\frac 2 3 }}+ \frac {1}{ |j|} \frac {1}{1+|t-\tfrac{\eta}{j}|},
\end{equation}
 we can argue analogously in the case $t\in I_{j,\eta}^L$ and conclude that the inequality \eqref{bd:deetaqNR} is true also in this case. In account of \eqref{eq:petaqR}, we thus obtain that for $\tfrac12\min(|\eta|^\frac23)\leq t\leq 2|\eta|$  one has
    \begin{equation}
\label{bd:deetaq}
    \big|\frac{\p_\eta q(t,k,\eta)}{q(t,k,\eta)}\big|\lesssim \frac 1 {\vert k\vert} +\frac{1}{|\eta|^{\frac 2 3 }}+ \frac {1}{1+|t-\frac{\eta}{j}|} .
\end{equation}
 Hence
\begin{align*}
    |q(t,k,\eta)-q(t,k,\xi)|&\leq |\eta-\xi||\p_\eta q(t,k,\xi')|\\
    &\lesssim \vert \eta-\xi\vert  \left(\frac{1}{|k|} +\frac{1}{|\eta|^{\frac 2 3 }}+ \frac {1}{ |j|} \frac 1 {1+|t-\frac{\xi'}{j}|}\right)q(t,k,\xi')
\end{align*}
for some $\xi'\in (\eta,\xi)$. For $\xi'\in  (\eta,\xi)$ we obtain by Lemma \ref{lem:fullq}
\begin{align*}
    \frac 1 {1+|t-\tfrac{\xi'}{j}\vert}&\le \langle  \eta -\xi \rangle  \frac 1 {1+\vert t-\frac \eta j\vert }\le\langle  \eta -\xi \rangle  \left( \sqrt{\tfrac {\p_t q }q }(\eta ) + \frac {|j|^3} {|\eta|} \right).
\end{align*}
Thus, using that $1\leq \vert j\vert \lesssim \vert \eta\vert^{\frac 1 3 }$
\begin{align*}
    |q(t,k,\eta)-q(t,k,\xi)|&\leq |\eta-\xi||\p_\eta q(t,k,\xi')|\\
    &\lesssim \vert\eta-\xi\vert \langle  \eta -\xi \rangle   \left(\frac{1}{|k|} +\frac{1}{|\eta|^{\frac 2 3 }}+ \frac {|j|^2}{|\eta|}+\sqrt{\tfrac {\p_t q }q }(\eta ) \right)q(t,k,\xi')\\
    &\lesssim \vert\eta-\xi\vert \langle  \eta -\xi \rangle  \left(\frac{1}{|k|} +\frac{1}{|\eta|^{\frac 1 3 }}+ \sqrt{\tfrac {\p_t q }q }(\eta ) \right)q(t,k,\xi').
\end{align*}
For the cases $t\in  I_{j,\eta }\setminus\tilde I_{j,\eta }$ we can argue in a similar way but the estimates are simpler since we do not have the term $(1+a_{j,\eta}\vert t-\tfrac \eta j \vert )$.

Therefore, thanks to Lemma \ref{lem:fullq} we conclude  that
\begin{equation}
   \frac{|q(t,k,\eta)-q(t,k,\xi)|}{q(t,k,\eta)}\lesssim \vert\eta-\xi\vert \langle  \eta -\xi \rangle \left(\frac{1}{|k|} +\frac{1}{|\eta|^{\frac 1 3 }}+ \sqrt{\tfrac {\p_t q }q }(\eta ) \right)e^{C|\eta-\xi|^\frac13}
\end{equation}
for some constant $C>0$. Putting all the estimates together we finally see that
\begin{align*}
    \frac{|J(t,k,\eta)-J(t,k,\xi)|}{J(t,k,\xi)}&\lesssim\vert\eta-\xi\vert \langle  \eta -\xi \rangle\big( \frac{1}{|k|} +\frac{1}{|\eta|^{\frac 1 3 }}+ \sqrt{\tfrac {\p_t q }q }(\eta )\big)e^{C|\eta-\xi|^\frac13}\\&\lesssim\vert\eta-\xi\vert \langle  \eta -\xi \rangle\big( \frac{1}{|k|^{\frac 1 3}} + \sqrt{\frac {\p_t q }q }(\eta )\big)e^{C|\eta-\xi|^\frac13}
\end{align*}
which is the desired result.
\end{proof}

\section{Fourier Multipliers and Gevrey Spaces}
\label{sec:appendix}

In this section, we summarize a few inequalities related to Gevrey spaces. In particular we use Lemma A.2 of \cite{Bedrossian15}:
\begin{lemma}\label{lem:useest}
    Let $0<s<1$ and $x\ge y\ge 0$
    \begin{itemize}
        \item[i)] If $x\neq 0$
        \begin{align}
            \vert x^s-y^s \vert &\lesssim \tfrac 1 {x^{1-s}+y^{1-s}}\vert x-y\vert
        \end{align}
         \item[ii)] If $\vert x-y \vert\le \tfrac x K $ and $K>1$
        \begin{align}
            \vert x^s-y^s \vert &\le \tfrac s {(K-1)^{1-s }}\vert x-y\vert^s
        \end{align}
        \item [iii)]In general,
        \begin{align}
            \vert x+y \vert^s &\le \left(\tfrac x{x+y} \right)^{1-s }(x^s+y^s)
        \end{align}
        in particular, if $y\le x\le K y $ for some $K$, then
        \begin{align}
            \vert x+y\vert^s \le \left(\tfrac K{1+K} \right)^{1-s }(x^s+y^s)
        \end{align}
    \end{itemize}

\end{lemma}

\bibliographystyle{alpha}

\bibliography{library}

\end{document}